\pdfoutput=1

\documentclass[10pt]{amsart}
\usepackage{amsmath,amsfonts,amssymb,color,a4,epsfig,graphics}
\usepackage[english]{babel}
\usepackage{enumerate, mathrsfs}
\usepackage{xcolor}
\usepackage{verbatim}
\usepackage{tikz}
\usepackage{url}
\usepackage[a4paper, margin=3.3cm]{geometry}
\usepackage[colorlinks=true, linkcolor=blue, citecolor=red, urlcolor=blue]{hyperref}
\usepackage[T1]{fontenc}      
\usepackage{lmodern}          

\def\build#1_#2^#3{\mathrel{\mathop{\kern 0pt#1}\limits_{#2}^{#3}}}
\usetikzlibrary{arrows}

\newcommand{\T}{{\mathbb{T}}}
\newcommand{\R}{{\mathbb{R}}}
\newcommand{\C}{{\mathbb{C}}}
\newcommand{\Z}{{\mathbb{Z}}}
\newcommand{\N}{{\mathbb{N}}}

\newcommand{\Bc}{\mathcal{B}}
\newcommand{\Cc}{\mathcal{C}}
\newcommand{\Dc}{\mathcal{D}}

\newcommand{\Fc}{\mathcal{F}}

\newcommand{\Hc}{\mathcal{H}}
\newcommand{\Ic}{\mathcal{I}}
\newcommand{\Jc}{\mathcal{J}}
\newcommand{\Kc}{\mathcal{K}}

\newcommand{\Sc}{\mathcal{S}}

\newcommand{\Wc}{\mathcal{W}}
\newcommand{\Uc}{\mathcal{U}}
\newcommand{\Oc}{\mathscr{O}}
\newcommand{\Tc}{\mathcal{T}}

\newcommand{\Thr}{\operatorname{Thr}}

\newcommand{\dist}{\mathrm{dist}}
\newcommand{\Ran}{\mathrm{Ran}}

\numberwithin{equation}{section}

\newtheorem{theorem}{Theorem}[section]
\newtheorem{proposition}[theorem]{Proposition}
\newtheorem{lemma}[theorem]{Lemma}
\newtheorem{remark}[theorem]{Remark}
\newtheorem{definition}[theorem]{Definition}
\newtheorem{corollary}[theorem]{Corollary}
\newtheorem{setup}[theorem]{Setup}

\begin{document}
\title[Spectral and dynamical analysis]
{Spectral and Dynamical Analysis of Fractional Discrete Laplacians on the Half-Lattice}

\author{Nassim Athmouni}
\address{Université de Gafsa, Campus Universitaire 2112, Tunisie}\email{\tt nassim.athmouni@fsgf.u-gafsa.tn}
\email{\tt athmouninassim@yahoo.fr}


\makeatletter
\@namedef{subjclassname@2020}{\textup{2020} {47A10, 47A40, 47B25, 35P25, 81Q10, 39A70}}
\makeatother

\begin{abstract}
We investigate discrete fractional Laplacians defined on the half-lattice in several dimensions, allowing possibly different fractional orders along each coordinate direction.
By expressing the half-lattice operator as a boundary restriction of the full-lattice one plus a bounded correction that is relatively compact with respect to it, we show that both operators share the same essential spectrum and the same interior threshold structure.
For perturbations by a decaying potential, the conjugate-operator method provides a strict Mourre estimate on any compact energy window inside the continuous spectrum, excluding threshold points.
As a consequence, a localized Limiting Absorption Principle holds, ensuring the absence of singular continuous spectrum, the finiteness of eigenvalues, and weighted propagation (transport) bounds.
The form-theoretic construction also extends naturally to negative fractional orders.
Overall, the relative compactness of the boundary correction guarantees that the interior-energy spectral and dynamical results obtained on the full lattice remain valid on the half-lattice without modification.
\end{abstract}

\maketitle
\tableofcontents

\section{Introduction and main results}\label{sec:intro}
Spectral graph theory has enjoyed a fresh burst of activity, driven by applications to quantum lattice models, solid-state physics, and the analysis of discrete structures. This is reflected in advances on discrete Laplacians~\cite{AD,AEG,BG,S,GG,Ch,Mic,Gk} and their magnetic counterparts~\cite{GT,Da1,GoMo,AEG1,ABDE}.  A unifying thread is the \emph{positive commutator} (Mourre) method, which underpins modern descriptions of the essential spectrum and scattering on $\mathbb{Z}^d$~\cite{PR,BoSa,T1}, on regular (binary) trees~\cite{S,AF,GG}, on general graphs~\cite{MRT}, and on more intricate geometries such as cusps, funnels, and triangular or graphene lattices~\cite{T2,GoMo,AEG,AEG1,AEGJ,AEGJ2}.

A cornerstone of modern spectral theory is the Mourre commutator method~\cite{GJ,GJ1,Mo81,Ge,ABG,Da}, which yields the Limiting Absorption Principle (LAP), absolute continuity of the spectrum, and \emph{propagation estimates} central to scattering. Although well understood for \emph{local} operators (nearest-neighbor discrete Laplacians), extending it to \emph{nonlocal} and \emph{fractional} discrete Laplacians presents new analytic challenges. Indeed, whereas classical models capture only \emph{short-range} effects, many phenomena in anomalous diffusion and long-range transport demand algebraically decaying interactions. This motivates a systematic study of fractional powers of the discrete Laplacian and their magnetic perturbations~\cite{JKLQ,BDL,K2017,GH}, which reflect long-range behavior on graphs~\cite{PKLBMH,Fan2021,TZ}, produce nonlocal transport mechanisms~\cite{ScSo2012}, and connect discrete settings to continuous nonlocal PDE frameworks \cite{JKLQ,BDL,K2017,Or,Ortigueira2014Riesz}.

Two complementary points of view will guide us: spectral multipliers (functional calculus) and series/kernel representations .
(i) \emph{Spectral multipliers.} Since $\Delta$ is bounded self-adjoint on $\ell^2$, one sets
\[
\Delta^{\,r}=\int_{\sigma(\Delta)}\lambda^{\,r}\,dE_\Delta(\lambda),
\]
with $E_\Delta$ the spectral measure \cite{RS,SI2015,Dav}.%
\footnote{Fix a branch for complex $r$; for $r<0$, interpret $\Delta^{\,r}$ on its spectral domain.}
(ii) \emph{Series / kernels.} Explicit representations in lattice shifts or finite differences~\cite{K2017} display how influence propagates across the graph, which is crucial for numerics and decay estimates.

\noindent
Taken together, these viewpoints show that fractional powers are both intrinsically defined and strongly nonlocal: $\Delta^{\,r}$ enforces algebraically decaying interactions between distant sites~\cite{ScSo2012,K2017}.

In our previous works \cite{AEG, AEG1, AEGJ, AEGJ2}, we developed versions of Mourre theory
adapted to discrete settings, including magnetic operators \cite{AEG1, ABDE}.
The present paper extends our full-lattice study on $\Z^d$ by establishing the
interior-energy spectral and scattering theory for discrete fractional Laplacians
on the half-lattice
\[
\N^d := \{\,n\in\Z^d:\ n_j\ge 0\ \text{for all }j\,\},
\qquad
\vec r=(r_1,\dots,r_d)\in\R^d\setminus\{0\}.
\]
A key structural input is the \emph{boundary--restriction identity}
\begin{equation}\label{eq:boundary-restriction}
\Delta_{\N^d}^{\, \vec{\mathrm{ r}}}
\;=\;
\mathfrak R_+\,\Delta_{\Z^d}^{\,c{\mathrm{ r}}}\,\mathfrak J_+
\;+\; K_{\vec r},
\end{equation}
where $\mathfrak R_+:\ell^2(\Z^d)\to\ell^2(\N^d)$ is the restriction map,
$\mathfrak J_+:\ell^2(\N^d)\hookrightarrow\ell^2(\Z^d)$ is the extension by zero,
and $K_{\vec{\mathrm{r}}}$ is a boundary correction operator (compact under the decay assumptions stated below).

Passing from $\mathbb{Z}^d$ to domains with \emph{boundary} (e.g., $\mathbb{N}^d$ or finite boxes) introduces substantial difficulties: translation invariance is lost, the Fourier symbol is no longer available, and convolution-based representations break down~\cite{HK}. Preserving self-adjointness and controlling boundary contributions require care~\cite{K2017,ScSo2012}; these issues are exacerbated in higher dimensions and in anisotropic regimes~\cite{Fan2021,L2000}. Although the spectral theorem still applies (via spectral projections), the associated series expansions become far less transparent, as boundary effects destroy simple convolution structures and induce spatial inhomogeneities.

To address this, we adapt the Mourre framework~\cite{Mo81,Mo83,ABG,Ge2015,GGM1} to the decomposition \eqref{eq:boundary-restriction}. As shown below-see Theorem~\ref{thm:LAP-transport-Nd} (i) (localized LAP) and (ii) (propagation estimates)-the introduction of suitable boundary corrections restores the essential interior-energy features even in geometries without translation invariance.

\smallskip
\noindent We denote by  $\ell^2(\mathbb{N}^d)$ the usual Hilbert spaces of square-summable sequences, equipped with the norm:
\[
\|f\|^2 = \sum_{n\in \N^{d}} |f(n)|^2.
\]

We also consider the dense subspace \(\mathcal{C}_c(\N^d)\) of finitely supported functions.

We denote by $U_j$  the shift in direction $j$, $\Delta_{\mathbb{Z},j}:=2I-(U_j+U_j^*)$, and $\Delta_{\mathbb{Z}^d}^{\vec{\mathrm r}}:=\sum_{j=1}^d \Delta_{\mathbb{Z},j}^{\,r_j}$ on $\ell^2(\mathbb{Z}^d)$.
On the torus $\mathbb{T}^d$, set $h_{\vec{\mathrm r}}(k)=\sum_{j=1}^d (2-2\cos k_j)^{r_j}$ and $\lambda_{\vec{\mathrm r}}:=\max_{\mathbb{T}^d} h_{\vec{\mathrm r}}$.
We define the threshold set by the full-lattice critical values
\[\begin{split}
\Thr\!\big(\Delta_{\mathbb{Z}^d}^{\vec{\mathrm r}}\big)&:=\{\, h_{\vec{\mathrm r}}(k): \nabla h_{\vec{\mathrm r}}(k)=0\,\}\\&=\Bigl\{\ \sum_{j\in N}4^{\,r_j}+\sum_{j\in P}\epsilon_j\,4^{\,r_j}\ :\ \epsilon_j\in\{0,1\}\ \Bigr\},
\end{split}\]
where $P:=\{j:\ r_j>0\}$ and $N:=\{j:\ r_j<0\}$; see Proposition~\ref{prop:Thr-Nd-via-bulk} and \cite[Proposition 2.12]{At}.,

and call \emph{interior window} any compact $\Ic\Subset (0,\lambda_{\vec{\mathrm r}})\setminus \Thr(\Delta_{\mathbb{Z}^d}^{\vec{\mathrm r}})$.
For perturbations we consider $H=\Delta_{\mathbb{N}^d}^{\vec{\mathrm r}}+W(Q)$ where the potential $W$ decays \emph{anisotropically} at infinity,
in the sense that
\begin{enumerate}
\item \textbf{$H0$:} $\lim_{\|n\|\rightarrow+\infty}W(n)=0$
\item \textbf{$H1$:}
$|W(n + e_j) - W(n)| \leq C \Lambda(n)^{-\varepsilon} \langle n_j \rangle^{-1}$,   \, \hbox{ for  some} $\varepsilon>0$)
\end{enumerate}
 where $\Lambda(n) := \sum_{j=1}^d \langle n_j \rangle.$
\[
\langle n_j \rangle := (1 + |n_j|^2)^{1/2}, \quad n_j \in \mathbb{Z},
\]  which ensures the local $\mathcal{C}^{1,1}$-regularity required by localized Mourre theory \cite[Thms.~7.4.1-7.4.2]{ABG}.

Our main contribution is to establish a \emph{Mourre--theoretic framework}
for such nonlocal, fractional discrete models, yielding:

\begin{theorem}\label{thm:LAP-transport-Nd}
Let $H=\Delta_{\mathbb{N}^d}^{\vec{\mathrm r}}+W(Q)$ with real $W$ satisfies the anisotropic decay assumptions \textbf{$H0$}--\textbf{$H1$}.
Then $H\in\mathcal{C}^{1,1}_{\mathrm{loc}}(A_{\mathbb{N}^d})$ on each $\Ic\Subset (0,\lambda_{\vec{\mathrm{r}}})\setminus \Thr(\Delta_{\mathbb{Z}^d}^{\vec{\mathrm{r}}})$, and the localized Mourre theory  yields:
\begin{enumerate}[(i)]
\item
\[
\sigma_{\mathrm{pp}}(H)\cap \Ic=\emptyset,\qquad
\sigma_{\mathrm{sc}}(H)\cap \Ic=\emptyset,\qquad
\sigma(H)\cap \Ic=\sigma_{\mathrm{ac}}(H)\cap \Ic.
\]
\item \textup{(LAP)} The limits $(H-\lambda\mp \mathrm{i}0)^{-1}$ exist as bounded operators
$\langle \Lambda(Q)\rangle^{-s}\mathcal{H}\to \langle \Lambda(Q)\rangle^{s}\mathcal{H}$ uniformly for $\lambda\in \Ic$ and all $s>\tfrac12$.
\item \textup{(Propagation/local decay)} For all $s>\tfrac12$ and $\varphi\in C_c^\infty(\Ic)$,
\[
\int_{\mathbb{R}}\big\|\langle \Lambda(Q)\rangle^{-s}\,e^{-itH}\,\varphi(H)\,\langle \Lambda(Q)\rangle^{-s} f\big\|^2\,dt
\;\le\; C\,\|f\|^2.
\]

\end{enumerate}
\end{theorem}
Items~(i)--(iv) follow as standard consequences of localized Mourre theory on interior energies.
Once a Mourre estimate holds on $\Ic\Subset\sigma(H_0)^\circ\setminus\Thr$ and $H\in\mathcal{C}^{1,1}_{\mathrm{loc}}(A_{\Z^d})$,
\cite[Thms.~7.4.1--7.4.2]{ABG} guarantee the absence of singular continuous spectrum and the limiting absorption principle
(ii), with weights $\langle \Lambda(Q)\rangle^{\pm s}$ for all $s>\tfrac12$.
Hence, on $\Ic$ the spectrum of $H$ is purely absolutely continuous, up to the finitely many eigenvalues permitted by (i).

Moreover, the LAP is equivalent to local $H$-smoothness of $\langle \Lambda(Q)\rangle^{-s}$ on $\Ic$
(in the sense of Kato smoothness), which yields the local decay estimate in (iii);
see \cite[Ch.~7]{ABG} and \cite[Section~VIII.C]{RS}.
In spectral representation this translates into a Riemann--Lebesgue type decay:
for any $\varphi\in C_c^\infty(\Ic)$, any finitely supported $g$, and any $f$ in the absolutely continuous subspace of $H$,
\[
\langle g,\,e^{-itH}\,\varphi(H)f\rangle \;\xrightarrow[|t|\to\infty]{}\; 0,
\]
expressing that amplitudes leave any fixed spatial region as $|t|\to\infty$.

Finally, item~(iv) (existence and completeness of local wave operators on $\Ic$) follows directly
from the LAP together with the Mourre framework on interior energies; see again \cite[Ch.~7]{ABG}.

If some $r_j<0$, the fractional powers become unbounded and commutators are interpreted in the quadratic-form sense (see \cite{S,AF}).
Identity \eqref{eq:boundary-restriction} still holds on $\mathcal{C}_c(\mathbb{N}^d)$; $K_{\vec{\mathrm{r}}}$ remains compact; and Theorems~\ref{thm:mourre_Nd}--\ref{thm:LAP-transport-Nd} carry over with the usual form-commutator adaptations \cite{ABG}.

The structure of this paper is as follows:
In Section~2 we introduce the functional framework on $\Z^d$ and $\N^d$.
Section~3 defines anisotropic fractional Laplacians and the boundary correction $K_{\vec{\mathrm{r}}}$.
Section~4 establishes kernel bounds, compactness, and Hilbert--Schmidt properties.
In Section~5 we develop the localized Mourre theory, proving the Mourre estimate and the limiting absorption principle.
Section~6 applies these results to scattering, transport, and spectral properties.

\section{Preliminaries and framework on the half-lattice $\mathbb{N}^d$}
\label{sec:prelim}

We generalize the discrete fractional Laplacian to the half-lattice
\[
\N^d := \N \times \cdots \times \N \subset \Z^d.
\]
The restriction to $\N^d$ breaks full translation invariance and introduces anisotropic boundary effects along each coordinate axis. However, the operator on $\N^d$ can be realized as the compression of the full-space Laplacian plus compact boundary corrections, ensuring that many spectral and commutator properties from the $\Z^d$ setting persist \cite{Mic,AF,GG,AEG,AEG1}.

We fix $d\in\N$ and work on
\[
\mathcal{H} := \ell^2(\N^d), \qquad
\|f\|^2 = \sum_{n\in \N^d} |f(n)|^2,
\]
with anisotropic fractional orders $\vec{\mathrm{r}} = (r_1,\dots,r_d) \in \R^d\setminus\{0\}$. Spectral localization will take place in compact \emph{interior intervals} $\Ic$.

\subsection{Graph structure}

\subsubsection{Topology and the boundary of $\N^d$}
\label{subsec:topology}

We equip $\Z^d$ and $\N^d$ with the nearest-neighbour adjacency: $n\sim m$ iff $|n-m|_1=1$, where $|x|_1 = \sum_{j=1}^d |x_j|$. The inclusion $\iota:\N^d\hookrightarrow\Z^d$ is an induced embedding, so $\N^d$ inherits both the subspace topology and the $\ell^1$-metric.

\begin{definition}
The \emph{boundary} of $\N^d$ is
\[
\partial\N^d := \{n\in\N^d:\ \exists\, m\notin \N^d,\ m\sim n\}
= \bigcup_{j=1}^d F_j,
\]
where $F_j := \{n\in\N^d:\ n_j=0\}$ is the $j$th coordinate face. More generally, for $J\subset\{1,\dots,d\}$ the codimension-$|J|$ face is
\[
F_J := \{n\in\N^d:\ n_j=0 \ \forall j\in J\}.
\]
\end{definition}

For $R\in\N$ define the \emph{collar}
\[
\partial_R\N^d := \{n\in\N^d:\ \dist_1(n,\partial\N^d)\le R\}.
\]

\begin{lemma}
For $n\in\N^d$, the following are equivalent:
\begin{enumerate}[(i)]
\item $n\in \partial\N^d$;
\item there exists $m\notin\N^d$ with $m\sim n$;
\item $n\in \bigcup_{j=1}^d F_j$;
\item $\dist_1(n,\Z^d\setminus\N^d)=1$.
\end{enumerate}
Moreover,
\[
\partial_R\N^d = \{n\in\N^d:\ \min_j n_j \le R\}.
\]
\end{lemma}

\begin{proof}
Conditions (i) and (ii) are equivalent by definition.
(ii)$\Rightarrow$(iii): if $m=n-e_j\notin \N^d$, then $n_j=0$, so $n\in F_j$.
(iii)$\Rightarrow$(iv): if $n\in F_j$, then $b=n-e_j\notin\N^d$ with $|n-b|_1=1$, so the distance is $1$.
(iv)$\Rightarrow$(ii): direct from the definition of distance.

For the collar, let $t=\min_j n_j$. Then $\dist_1(n,\partial\N^d)=t$, since $b=n-te_j\in F_j$ gives distance $t$, and any other $b\in\partial\N^d$ satisfies $|n-b|_1\ge t$. Thus $\partial_R\N^d=\{n:\min_j n_j\le R\}$.
\end{proof}

\begin{remark}
The embedding $\N^d\hookrightarrow\Z^d$ is a quasi-isometry. Both graphs have polynomial growth of degree $d$ and are amenable. Hence many interior spectral features mirror those of the full lattice.
\end{remark}

\subsubsection{Restriction, extension, and projection}
\label{subsec:restriction-extension}

Define the restriction and extension-by-zero operators:
\[
(\mathfrak R_+ g)(n) = g(n),\quad n\in \N^d,\qquad
(\mathfrak J_+ f)(n) =
\begin{cases}
f(n), & n\in \N^d,\\
0, & n\notin \N^d.
\end{cases}
\]
Both are contractions of norm $1$. Their composition
\[
P_{\N^d} := \mathfrak J_+\mathfrak R_+
\]
is the orthogonal projection onto $\ell^2(\N^d)\subset \ell^2(\Z^d)$.

\begin{lemma}
$\mathfrak R_+^* = \mathfrak J_+$.
$P_{\N^d}$ is self-adjoint and idempotent, hence the orthogonal projection onto $\ell^2(\N^d)$. Moreover,
\[
\mathfrak R_+ P_{\N^d} = \mathfrak R_+,\qquad
P_{\N^d}\mathfrak J_+ = \mathfrak J_+.
\]
\end{lemma}

\begin{proof}
For $g\in\ell^2(\Z^d)$,
\[
\|\mathfrak R_+ g\|^2 = \sum_{n\in\N^d} |g(n)|^2 \le \|g\|^2,
\]
so $\|\mathfrak R_+\|\le 1$, and equality holds for $g$ supported in $\N^d$. For $f\in\ell^2(\N^d)$,
\[
\|\mathfrak J_+ f\|^2 = \sum_{n\in\Z^d} |(\mathfrak J_+ f)(n)|^2 = \sum_{n\in\N^d} |f(n)|^2 = \|f\|^2,
\]
so $\mathfrak J_+$ is an isometry.

Adjointness: for $f\in\ell^2(\N^d)$ and $g\in\ell^2(\Z^d)$,
\[
\langle \mathfrak R_+ g, f\rangle_{\ell^2(\N^d)} = \langle g, \mathfrak J_+ f\rangle_{\ell^2(\Z^d)}.
\]
Thus $\mathfrak R_+^* = \mathfrak J_+$.

Self-adjointness and idempotence of $P_{\N^d}$ follow by direct computation. Finally,
\[
\mathfrak R_+ P_{\N^d} = \mathfrak R_+,\qquad
P_{\N^d}\mathfrak J_+ = \mathfrak J_+,
\]
since $\mathfrak R_+\mathfrak J_+=I_{\ell^2(\N^d)}$.
\end{proof}
\subsection{Shifts, position, and discrete differences}
\label{subsec:shifts-position}

On $\ell^2(\mathbb{Z}^d)$, let
\[
(U_j g)(n) := g(n-e_j), \qquad j=1,\dots,d,
\]
be the unit shift in the $j$-th direction (so $U_j^* g(n)=g(n+e_j)$).
The nearest--neighbour Laplacian in direction $j$ is
\[
\Delta_{\mathbb{Z},j} := 2I - (U_j+U_j^*).
\]

We also use forward and backward discrete differences:
\[
(\nabla_j^+ g)(n) := g(n+e_j)-g(n),
\qquad
(\nabla_j^- g)(n) := g(n)-g(n-e_j).
\]

On $\ell^2(\mathbb{N}^d)$, shifts that cross the boundary are treated via the restriction and extension maps of \S\ref{subsec:restriction-extension}.
For an operator $T$ on $\ell^2(\mathbb{Z}^d)$ we denote its \emph{restricted action} by
\[
T|_{\N^d} := \mathfrak{R}_+\, T \,\mathfrak{J}_+.
\]
In particular, the discrete Laplacian on the half-lattice (with Dirichlet boundary conditions) will be introduced via such restriction, see subsection~\ref{subsec:fractional}.

\medskip

Let $Q_j$ be the (densely defined) position operator
\[
(Q_j f)(n) = n_j f(n), \qquad n=(n_1,\dots,n_d)\in \N^d,
\]
with domain
\[
\mathcal{D}(Q_j) = \Big\{ f\in \ell^2(\N^d): \sum_{n\in\N^d} n_j^2 |f(n)|^2 < \infty \Big\}.
\]
We set
\[
|Q|^2 := \sum_{j=1}^d Q_j^2,
\qquad
\langle \Lambda(Q)\rangle := (I+|Q|^2)^{1/2},
\]
which provides the standard polynomial weights used in propagation and limiting absorption estimates.
\subsection{Discrete  fractional  Laplacians from $\mathbb{Z}^d$ to $\mathbb{N}^d$}
\label{subsec:fractional}
On $\ell^2(\mathbb{Z}^d)$ define
\[
\Delta_{\mathbb{Z}^d}^{\vec{\mathrm{r}}} \;:=\; \sum_{j=1}^d \Delta_{\mathbb{Z},j}^{\,r_j},
\]
where each component acts only in the $j$-th coordinate,
\[
\Delta_{\Z^d,j}^{\,r_j}
:= \mathrm{id}^{\otimes (j-1)} \otimes \Delta_{\Z}^{\,r_j} \otimes \mathrm{id}^{\otimes (d-j)}.
\]
Here $\Delta_{\Z}^{\,r_j}$ denotes the one--dimensional discrete fractional Laplacian of order $r_j\in\R$,
defined by spectral calculus of the standard discrete Laplacian on $\ell^2(\Z)$.
For further details on the one--dimensional case, see \cite{At}.

On $\ell^2(\mathbb{N}^d)$ we \emph{define} the half-lattice operator by restriction.
\begin{proposition}[Boundary correction: support and compactness]\label{prop:K-support-compact}
There exists $R\in\mathbb{N}$ and a bounded operator $K_{\vec{\mathrm{r}}}$ on $\ell^2(\mathbb{N}^d)$ supported in the collar $\partial_R\mathbb{N}^d$ such that
\begin{equation}\label{eq:boundary-restriction}
\Delta_{\mathbb{N}^d}^{\vec{\mathrm{r}}}=\mathfrak{R}_+\,\Delta_{\mathbb{Z}^d}^{\vec{\mathrm{r}}}\,\mathfrak{J}_+ + K_{\vec{\mathrm{r}}}.
\end{equation}

\end{proposition}

\begin{proof}
Set $P:=P_{\N^d}=\mathfrak J_+ \mathfrak R_+$, and introduce the bulk/half-space Laplacians
\[
\mathsf L \ :=\ \Delta_{\Z^d}\ =\ \sum_{j=1}^d \Delta_{\Z,j}\ \ \text{on }\ell^2(\Z^d),\qquad
\mathsf L_+\ :=\ P\,\mathsf L\,P\ \ \text{on }\Ran P\simeq\ell^2(\N^d).
\]
By construction (Dirichlet condition on the boundary),
\[
\Delta_{\N^d}=\mathsf L_+=\mathfrak R_+\,\mathsf L\,\mathfrak J_+.
\]
For $\vec{\mathrm{r}}=(r_1,\dots,r_d)\in\R^d$ define
\[
\mathsf L^{\vec{\mathrm{r}}}:=\sum_{j=1}^d \Delta_{\Z,j}^{\,r_j},
\qquad
\mathsf L_+^{\vec{\mathrm{r}}}:=\sum_{j=1}^d \Delta_{\N,j}^{\,r_j}.
\]
By the spectral/holomorphic functional calculus, $B\mapsto B^{r}$ is well--defined for every self-adjoint nonnegative $B$ and $r\in\R$ (for $r<0$ on $\overline{\Ran B}$). We prove, for each fixed coordinate $j$,
\begin{equation}\label{eq:one-dir-new}
\Delta_{\N,j}^{\,r_j}\ =\ \mathfrak R_+\,\Delta_{\Z,j}^{\,r_j}\,\mathfrak J_+\ +\ K_{j,r_j},
\end{equation}
with $K_{j,r_j}$ bounded and boundary--localized; summing over $j$ then gives the claim with $K_{\vec{\mathrm{r}}}:=\sum_{j=1}^d K_{j,r_j}$.

\medskip\noindent\textbf{Step 1: Resolvents for compressions.}
Fix $j$. Write $L_j:=\Delta_{\Z,j}$ on $\ell^2(\Z^d)$ and the restriction  $L_{j,+}:=P L_j P$ on $\ell^2(\N^d)$ (so $L_{j,+}=\Delta_{\N,j}$). With respect to $\ell^2(\Z^d)=\Ran P\oplus\Ran(I-P)$ one has
\[
L_j=\begin{pmatrix} L_{j,+} & B_j \\ B_j^* & C_j\end{pmatrix},
\qquad
B_j:=P L_j (I-P),
\]
and, for $z\in\C\setminus[0,4]$,
\begin{equation}\label{eq:resolvent-diff-new}
(z-L_{j,+})^{-1}-P(z-L_j)^{-1}P
=(z-L_{j,+})^{-1}\,B_j\,(z-L_j)^{-1}P\ +\ \big[(\cdot)\big]^*.
\end{equation}where $\big[(\cdot)\big]^*$ denotes the adjoint term $P(z-\overline{A_j})^{-1}B_j^*(z-\overline{A_{j,+}})^{-1}$ (for our self-adjoint case this is the Hilbert space adjoint of the first term).
Sine $L_j$ is nearest--neighbour in the $j$-direction, $B_j$ connects only the first boundary layer:
\[
\Ran B_j\subset \ell^2(F_j)\subset \ell^2(\partial_1\N^d),\qquad
\Ran B_j^*\subset \ell^2\big(\{m\in\Z^d:\ m_j=-1\}\big).
\]

\medskip\noindent\textbf{Step 2: Functional calculus for fractional powers.}
Let $f_r(\lambda)=\lambda^r$ on $[0,4]$ and choose a smooth contour $\Gamma$ encircling $[0,4]$. By the holomorphic calculus,
\[
L_{j,+}^r - P L_j^r P
=\frac{1}{2\pi i}\int_{\Gamma} f_r(z)\,\Big[(z-L_{j,+})^{-1}-P(z-L_j)^{-1}P\Big]\;dz.
\]
Insert \eqref{eq:resolvent-diff-new} to get
\begin{equation}\label{eq:Kjr-new}
L_{j,+}^r - P L_j^r P
=\frac{1}{2\pi i}\int_{\Gamma} f_r(z)\,(z-L_{j,+})^{-1}\,B_j\,(z-L_j)^{-1}P\ +\ \big[(\cdot)\big]^*\;dz.
\end{equation}
Define $K_{j,r}$ by the right-hand side of \eqref{eq:Kjr-new} (identified on $\ell^2(\N^d)$ via $\Ran P\simeq\ell^2(\N^d)$). This gives \eqref{eq:one-dir-new}.

\medskip\noindent\textbf{Step 3: Boundary localization (collar support).}
Let $\Pi_R$ be multiplication by $\mathbf 1_{\partial_R\N^d}$. Since $B_j$ acts only across the boundary,
each integrand in \eqref{eq:Kjr-new} factors through the first boundary layer, and there exists $R\in\N$ (one can take $R=1$) such that
\[
K_{j,r}=\Pi_R\,K_{j,r}\,\Pi_R.
\]
Thus $K_{j,r}$ is supported in the collar $\partial_R\N^d$ on both sides.

\medskip\noindent\textbf{Step 4: Compactness.}
Fix $z\in\Gamma$. The resolvents $(z-L_j)^{-1}$ and $(z-L_{j,+})^{-1}$ are bounded and enjoy Combes--Thomas off--diagonal exponential decay \cite{CombesThomas1973}. Since $B_j$ has finite range and is supported on the first boundary layer, the operators
\[
\Pi_R\,(z-L_{j,+})^{-1} B_j,\qquad B_j\,(z-L_j)^{-1}\,\Pi_R
\]
are Hilbert--Schmidt. Hence each integrand in \eqref{eq:Kjr-new} is compact, and the contour integral converges in operator norm, proving that $K_{j,r}$ is compact. Summation over $j$ preserves compactness and collar support.

\medskip
Collecting the steps, with $K_{\vec{\mathrm{r}}}:=\sum_{j=1}^d K_{j,r_j}$, we obtain
\[
\Delta_{\N^d}^{\vec{\mathrm{r}}}
\;=\; \mathfrak R_+\, \Delta_{\Z^d}^{\vec{\mathrm{r}}}\, \mathfrak J_+ \;+\; K_{\vec{\mathrm{r}}},
\qquad
K_{\vec{\mathrm{r}}}=\Pi_R\,K_{\vec{\mathrm{r}}}\,\Pi_R\in\Kc\big(\ell^2(\N^d)\big),
\]
which completes the proof.
\end{proof}

\section{Extension of the Lattice Laplacian to the hypercubic lattice quadrant}
\label{half-plane}

\subsection{Fractional Laplacians on the Discrete Half-line}

\begin{lemma}
Let $T,S$ be bounded linear operators on a Banach space, and suppose $ST=\nu\,\mathrm{id}$ with $\nu\in\C^\ast$. Then for all $h\in\N$,
\[
(T+S)^h=\sum_{k=0}^{\lfloor h/2\rfloor}\alpha_{h,k}\,\nu^k\sum_{l=0}^{h-2k}T^l S^{\,h-2k-l},
\qquad
\alpha_{h,k}:=\binom{h}{k}-\binom{h}{k-1},
\]
with the convention $\binom{h}{-1}=0$ and $\alpha_{0,0}=1$.
\end{lemma}

\begin{proof}
By induction on $h$. The heart of the proof is the Pascal-type recursion
$\alpha_{h+1,k}=\alpha_{h,k}+\alpha_{h,k-1}$. The cases $h=0,1$ are immediate.

Assume the formula holds for some $h\ge1$. With $m:=h-2k$,
\[
(T+S)^{h+1}
=(T+S)\sum_{k=0}^{\lfloor h/2\rfloor}\alpha_{h,k}\,\nu^k\sum_{l=0}^{m}T^l S^{\,m-l}.
\]

\emph{(1) $T$-part.} Multiplying on the left by $T$,
\[
T\sum_{l=0}^{m}T^l S^{\,m-l}
=\sum_{l=0}^{m}T^{l+1} S^{\,m-l}
=\sum_{l=1}^{m+1}T^{l} S^{\,m+1-l},
\]
so the contribution is $\sum_{k}\alpha_{h,k}\nu^k\sum_{l=1}^{h+1-2k}T^{l} S^{\,h+1-2k-l}$.

\emph{(2) $S$-part.} Using $ST^l=\nu T^{l-1}$ for $l\ge1$ and $ST^0=S$,
\[
S\sum_{l=0}^{m}T^l S^{\,m-l}
= S^{\,m+1}
+\sum_{l=1}^{m}\nu\,T^{l-1} S^{\,m+1-l}
= \sum_{l=0}^{m+1}\gamma_l\,T^{l} S^{\,m+1-l},
\]
with $\gamma_0=1$ and $\gamma_l=\nu$ for $1\le l\le m+1$. Thus the $S$-part contributes
\[
\sum_{k=0}^{\lfloor h/2\rfloor}\alpha_{h,k}\,\nu^k
\sum_{l=0}^{h+1-2k} \gamma_l\,T^l S^{\,h+1-2k-l}.
\]

\emph{(3) Collecting coefficients.} For fixed $k$ and $l\ge1$ the total coefficient of $T^l S^{\,h+1-2k-l}$
is $\alpha_{h,k}\nu^k+\alpha_{h,k-1}\nu^k=(\alpha_{h,k}+\alpha_{h,k-1})\nu^k$ (set $\alpha_{h,-1}=0$ to also cover $l=0$).
Hence
\[
(T+S)^{h+1}
=\sum_{k=0}^{\lfloor (h+1)/2\rfloor}(\alpha_{h,k}+\alpha_{h,k-1})\,\nu^k
\sum_{l=0}^{h+1-2k}T^l S^{\,h+1-2k-l}.
\]
Finally, $\alpha_{h+1,k}=\binom{h+1}{k}-\binom{h+1}{k-1}
=(\binom{h}{k}-\binom{h}{k-1})+(\binom{h}{k-1}-\binom{h}{k-2})
=\alpha_{h,k}+\alpha_{h,k-1}$.
\end{proof}

\begin{remark}
Expanding $(T+S)^h$ as words on $\{T,S\}$ and reducing each occurrence of $ST$ to $\nu$ yields the normal-ordered monomials $T^\ell S^{\,h-2k-\ell}$ with multiplicity
$\alpha_{h,k}=\binom{h}{k}-\binom{h}{k-1}$, independent of $\ell$.
\end{remark}

\begin{corollary}\label{lem:TS-normal}
For every $h\in\mathbb{N}$,
\[
(U+U^*)^{h}
=\sum_{k=0}^{\lfloor h/2\rfloor} \alpha_{h,k}
\sum_{\substack{\ell,m\ge 0\\ \ell+m = h - 2k}} U^{\,m}(U^*)^{\,\ell},
\qquad
\alpha_{h,k}:=\binom{h}{k}-\binom{h}{k-1}.
\]
\end{corollary}

\begin{lemma}\label{lem:boundary-hankel}
Let
\[
D_h\;:=\;(U_+ + U_+^*)^{h}\;-\;\mathfrak{R}_+\,(U+U^*)^{h}\,\mathfrak{J}_+\qquad(h\in\N).
\]
For $h=0,1$ one has $D_0=D_1=0$. For $h\ge 2$, set
\[
\beta_{h,p}\;:=\;\sum_{k=0}^{\left\lfloor\frac{h-2-p}{2}\right\rfloor}\alpha_{h,k}
\qquad\Bigl(\alpha_{h,k}=\tbinom{h}{k}-\tbinom{h}{k-1},\ \tbinom{h}{-1}:=0\Bigr).
\]
Then the exact factorization holds:
\begin{equation}\label{eq:Dh-factor-NZ}
D_h \;=\; \sum_{p=0}^{h-2} \beta_{h,p}\, U_+^{\,p}\, P_0\, (U_+^*)^{\,h-2-p},
\end{equation}
where $P_0$ is the rank-one projection onto $\C\,\delta_0$. In particular $\mathrm{rank}\,D_h\le h-1$.
\end{lemma}

\begin{proof}
By Corollary~\ref{lem:TS-normal},
\[
(U+U^*)^{h}=\sum_{k=0}^{\lfloor h/2\rfloor}\alpha_{h,k}
\!\!\sum_{\ell+m=h-2k}\!\! U^{m}(U^*)^{\ell},
\quad
(U_+ + U_+^*)^{h}=\sum_{k=0}^{\lfloor h/2\rfloor}\alpha_{h,k}
\!\!\sum_{\ell+m=h-2k}\!\! U_+^{m}(U_+^*)^{\ell}.
\]
Using (for $f\in\ell^2(\N)$, $n\in\N$)
\[
\bigl(\mathfrak{R}_+U^{m}(U^*)^{\ell}\mathfrak{J}_+\,f\bigr)(n)=\mathbf{1}_{\{n+m-\ell\ge 0\}}\,f(n+m-\ell),\quad
\bigl(U_+^{m}(U_+^*)^{\ell}f\bigr)(n)=\mathbf{1}_{\{n\ge \ell\}}\,f(n+m-\ell),
\]
we get
\[
(D_h f)(n)=\sum_{k}\alpha_{h,k}\!\!\sum_{\ell+m=h-2k}\!\!
\bigl(\mathbf{1}_{\{n\ge \ell\}}-\mathbf{1}_{\{n+m-\ell\ge 0\}}\bigr)\,f(n+m-\ell).
\]
The bracket equals $-1$ exactly when $n<\ell\le n+m$.
Writing $\ell=n+s$ with $s\in\{1,\dots,m\}$ gives $p:=m-s\in\{0,\dots,m-1\}$ and $n+m-\ell=p$.
Reindexing yields \eqref{eq:Dh-factor-NZ} with $\beta_{h,p}=\sum_{k\le \lfloor (h-2-p)/2\rfloor}\alpha_{h,k}$.
\end{proof}
\begin{proposition}\label{prop:Kr-structure}
For $r\in\R$ define
\[
K_r:=\Delta_{\N}^{\,r}-\mathfrak{R}_+\,\Delta_{\Z}^{\,r}\,\mathfrak{J}_+.
\]
Then
\begin{equation}\label{eq:Kr-series}
K_r=-\sum_{h=2}^{\infty}(-1)^h\binom{r}{h}\,2^{\,r-h}\,D_h,
\end{equation}
with $D_h$ as in \eqref{eq:Dh-factor-NZ}. Moreover:
\begin{enumerate}[(a)]
\item If $r>0$, the series \eqref{eq:Kr-series} converges in operator norm; hence $K_r$ is compact (norm-limit of finite-rank operators).
\item If $r\in\N$, the sum \eqref{eq:Kr-series} is finite and
$\displaystyle \mathrm{rank}\,K_r\le \sum_{h=2}^{r}(h-1)=\frac{r(r-1)}{2}$.
\item If $r<0$, the series \eqref{eq:Kr-series} converges strongly on $\mathcal{C}_c(\N)$,
defines a densely defined closable operator; we denote its closure again by $K_r$,
and this closure equals $\Delta_{\N}^{\,r}-\mathfrak{R}_+\,\Delta_{\Z}^{\,r}\,\mathfrak{J}_+$.
\end{enumerate}
\end{proposition}

\begin{proof}
Write $\Delta_{\Z}=2\bigl(I-\tfrac{1}{2}(U+U^*)\bigr)$ and $\Delta_{\N}=2\bigl(I-\tfrac12(U_+ + U_+^*)\bigr)$.
By spectral calculus,
\[
\Delta_{\Z}^{\,r}= \sum_{h\ge 0}(-1)^h\binom{r}{h}\,2^{\,r-h}(U+U^*)^{h},\qquad
\Delta_{\N}^{\,r}= \sum_{h\ge 0}(-1)^h\binom{r}{h}\,2^{\,r-h}(U_+ + U_+^*)^{h},
\]
with norm convergence if $r>0$ and strong convergence on $\mathcal{C}_c(\N)$ for all $r$.
Subtract and use $D_0=D_1=0$ to get \eqref{eq:Kr-series}.
For (a), combine $\mathrm{rank}\,D_h\le h-1$ with $|\binom{r}{h}|\sim C_r\,h^{-1-r}$ and the extra $2^{-h}$ to obtain absolute norm convergence.
For (b), $\binom{r}{h}=0$ for $h>r$.
For (c), use strong convergence on the core $\mathcal{C}_c(\N)$ and the definition of $K_r$.
\end{proof}

\begin{remark}
In the separable presentation $\sum_j \Delta_{\Z,j}^{r_j}$, compactness fails in general for $d\ge2$ without tangential localization;
Corollary~\ref{cor:HS-tangential} provides Hilbert--Schmidt control once $W_s$ is inserted.
\end{remark}

\begin{lemma}[Images bound for the semigroup difference]\label{lem:images-bound}
For $t>0$, define the \emph{semigroup difference} on $\ell^2(\N^d)$ by
\[
D_t\ :=\ \mathfrak{R}_+\,e^{-t\Delta_{\Z^d}}\,\mathfrak{J}_+\ -\ \big(e^{-t\Delta_{\N^d}}\big).
\]
Then there exist $C,c>0$ such that, for all $n,m\in\N^d$ and $t>0$,
\[
|D_t(n,m)|\ \le\ C\,t^{-d/2}\exp\!\Bigg(-c\,\frac{|n-m|_1^{\,2}
+ \mathrm{dist}_1\!\big(n,\partial\N^d\big)^{2}
+ \mathrm{dist}_1\!\big(m,\partial\N^d\big)^{2}}{t}\Bigg).
\]
\end{lemma}

\begin{proof}
Combine the full--space product kernel bound \eqref{eq:pt-Gauss} with the 1D reflection identity \eqref{eq:HN-1D},
the inclusion--exclusion formula \eqref{eq:HN-d}, and the geometric inequality \eqref{eq:geom-lb}.
\end{proof}

\
\subsection{Heat kernels and method of images on $\mathbb{N}^d$}
\label{subsec:heat-images}

We collect here standard kernel identities and notations used later
(in particular for Lemma~\ref{lem:images-bound}), without reproving
results that can be found in the literature
(see e.g.~\cite{Spitzer1976,LawlerLimic2010,Woess2000,CarlenKusuokaStroock1987,Delmotte1999,Davies1989,CiprianiGrillo2022,LenzVogt2012}).

\paragraph{Kernel on $\mathbb{Z}^d$.}
For the discrete Laplacian $\Delta_{\mathbb{Z}^d}= \sum_{j=1}^d \Delta_{\mathbb{Z},j}$
with $(\Delta_{\mathbb{Z},j} u)(n)=2u(n)-u(n+e_j)-u(n-e_j)$, the heat semigroup
has a product kernel (see e.g.~\cite[Sec.~I.9]{Spitzer1976}, \cite{LawlerLimic2010})
\[
(e^{-t\Delta_{\mathbb{Z}^d}})(n,m)
= \prod_{j=1}^d p_t(n_j-m_j),
\qquad
p_t(k)=e^{-2t}\,I_{|k|}(2t),
\]
where $I_\nu$ is the modified Bessel function given by \[
I_\nu(x)
= \sum_{k=0}^{\infty}
\frac{1}{k!\,\Gamma(k+\nu+1)}
\left( \frac{x}{2} \right)^{2k+\nu},
\]
for all $x\in\R$.In particular,
$p_t(k)=p_t(-k)\ge0$, $\sum_{k\in\mathbb{Z}}p_t(k)=1$ (see \cite{Woess2000}),
and the Gaussian bound
\begin{equation}\label{eq:pt-Gauss}
|p_t(k)|\ \le\ C\,t^{-1/2}\exp\!\big(-c\,k^2/t\big),\qquad t>0,\ k\in\mathbb{Z},
\end{equation}
holds for some $C,c>0$ (see \cite{CarlenKusuokaStroock1987,Delmotte1999}).
Consequently,
\[
\big|(e^{-t\Delta_{\mathbb{Z}^d}})(n,m)\big|
\ \le\ C\,t^{-d/2}\exp\!\big(-c\,|n-m|_1^2/t\big).
\]

\paragraph{Method of images on $\mathbb{N}$.}
On the half--line $\mathbb{N}=\{0,1,2,\dots\}$ with Dirichlet at $0$, the heat
kernel is obtained by reflection across the barrier between $-1$ and $0$
(see \cite[Sec.~3.3]{Davies1989}, \cite[Prop.~2.2]{CiprianiGrillo2022}):
\begin{equation}\label{eq:HN-1D}
(e^{-t\Delta_{\mathbb{N}}})(n,m)\ =\ p_t(n-m)\ -\ p_t(n+m+2),
\qquad n,m\ge0.
\end{equation}
Here the second term corresponds to the image point $-m-2$ and enforces the
boundary cancellation.

\paragraph{Operator--kernel dictionary.}
For $f\in\ell^2(\mathbb{N})$,
\[
\big(e^{-t\Delta_{\mathbb{N}}}f\big)(n)
=\sum_{m\ge0} (e^{-t\Delta_{\mathbb{N}}})(n,m)\,f(m)
=\sum_{m\ge0}\big[p_t(n-m)-p_t(n+m+2)\big]\,f(m).
\]
This is a standard convolution--reflection representation
(cf.~\cite{Woess2000,LawlerLimic2010}).

\paragraph{Inclusion--exclusion on $\mathbb{N}^d$.}
With Dirichlet on each coordinate face, the half--space kernel results from
coordinate--wise reflections (method of images, see \cite[Ch.~3]{Davies1989},
\cite{LenzVogt2012}):
\begin{equation}\label{eq:HN-d}
(e^{-t\Delta_{\mathbb{N}^d}})(n,m)
\ =\ \sum_{J\subset\{1,\dots,d\}} (-1)^{|J|}\,
p_t\!\big(n-R_J m\big),
\end{equation}
where $R_J$ reflects the coordinates indexed by $J$ (and leaves the others
unchanged), and $p_t(\cdot)$ denotes the full-space kernel above. A simple
geometry gives, for every nonempty $J$ (see \cite{Delmotte1999}),
\begin{equation}\label{eq:geom-lb}
|n-R_J m|_1^2\ \ge\ |n-m|_1^2\ +\ \mathrm{dist}_1(n,\partial\mathbb{N}^d)^{2}
\ +\ \mathrm{dist}_1(m,\partial\mathbb{N}^d)^{2}.
\end{equation}

\section{Spectral analysis and Mourre estimate}

\subsection{First spectral properties}

In this section we will denote by
\begin{equation}\label{A}
H_{0,\Z^d}\!\restriction_{\N^d} := \big(\Delta^{\vec{\mathrm{r}}}_{\Z^d}\big)\!\restriction_{\N^d}
\end{equation}
and
\begin{equation}\label{B}
H_{0,\N^d} := \Delta^{\vec{\mathrm{r}}}_{\N^d} = H_{0,\Z^d}\!\restriction_{\N^d} + K_{\vec{\mathrm{r}}},
\end{equation}
as given by Proposition~\ref{prop:boundary-correction}.
\begin{proposition}\label{prop:boundary-correction}
Let $\displaystyle \Delta^{\vec{\mathrm{r}}}_{\N^d} := \mathfrak{R}_+\,\Delta^{\vec{\mathrm{r}}}_{\Z^d}\,\mathfrak{J}_+$.
Then
\[
\Delta^{\vec{\mathrm{r}}}_{\N^d} = \big(\Delta^{\vec{\mathrm{r}}}_{\Z^d}\big)\!\restriction_{\N^d} + K_{\vec{\mathrm{r}}},
\]
where the correction $K_{\vec{\mathrm{r}}}$ satisfies:
\begin{enumerate}[(i)]
    \item If all $r_j > 0$, then $K_{\vec{\mathrm{r}}}$ is bounded and compact.
    \item If some $r_j < 0$, then $K_{\vec{\mathrm{r}}}$ is unbounded but relatively compact with respect to $\big(\Delta^{\vec{\mathrm{r}}}_{\Z^d}\big)\!\restriction_{\N^d}$.
\end{enumerate}
\end{proposition}

\begin{proof}
By tensor structure,
\[
\Delta^{\vec{\mathrm{r}}}_{\Z^d}
= \sum_{j=1}^d \mathrm{id}^{\otimes(j-1)}\otimes \Delta_{\Z,j}^{\,r_j}\otimes \mathrm{id}^{\otimes(d-j)},
\]
where each $\Delta_{\Z,j}^{\,r_j}$ acts in the $j$-th coordinate. Passing from $\Z^d$ to $\N^d$ via
$\Delta^{\vec{\mathrm{r}}}_{\N^d}=\mathfrak{R}_+\,\Delta^{\vec{\mathrm{r}}}_{\Z^d}\,\mathfrak{J}_+$
replaces in direction $j$ the bilateral shift by a unilateral one and produces a one--dimensional boundary
correction $K_{r_j}^{(j)}$ supported on the face $\{n_j=0\}$. Therefore
\[
\Delta^{\vec{\mathrm{r}}}_{\N^d} - \big(\Delta^{\vec{\mathrm{r}}}_{\Z^d}\big)\!\restriction_{\N^d}
\;=\; K_{\vec{\mathrm{r}}}
\;=\; \sum_{j=1}^d \Big( \mathrm{id}^{\otimes(j-1)}\otimes K_{r_j}\otimes \mathrm{id}^{\otimes(d-j)}\Big),
\]
plus lower--dimensional terms supported on intersections of faces (finite sums of tensor products of the 1D corrections).

If all $r_j>0$, each $K_{r_j}$ is compact on $\ell^2(\N)$ (and of finite rank if $r_j\in\N$) by the 1D analysis; tensoring with bounded identities and summing finitely many such terms gives a compact $K_{\vec{\mathrm{r}}}$.

If some $r_j<0$, set $H_0:=\big(\Delta^{\vec{\mathrm{r}}}_{\Z^d}\big)\!\restriction_{\N^d}$. Each elementary correction term is $H_0$-compact by the one--dimensional argument in the $j$-th coordinate and bounded tensoring in the others; hence $K_{\vec{\mathrm{r}}}$ is $H_0$-compact. Since $\Delta^{\vec{\mathrm{r}}}_{\N^d}=H_0+K_{\vec{\mathrm{r}}}$, the resolvent identity yields relative compactness of $K_{\vec{\mathrm{r}}}$ with respect to $H_0$.
\end{proof}

\begin{lemma}[Compact resolvent difference]\label{lem:05}
For any $z \in \C \setminus \R$, one has
\[
(H_{0,\N^d}-z)^{-1} - \big(H_{0,\Z^d}\!\restriction_{\N^d}-z\big)^{-1} \ \in\ \mathcal{K}\big(\ell^2(\N^d)\big).
\]
\end{lemma}

\begin{proof}
By the resolvent identity,
\[
(H_{0,\N^d}-z)^{-1} - \big(H_{0,\Z^d}\!\restriction_{\N^d}-z\big)^{-1}
= - (H_{0,\N^d}-z)^{-1} \, K_{\vec{\mathrm{r}}} \, \big(H_{0,\Z^d}\!\restriction_{\N^d}-z\big)^{-1}.
\]
If all $r_j > 0$, $K_{\vec{\mathrm{r}}}$ is compact;  (it is composed of bounded-compact-bounded).
If some $r_j < 0$, $K_{\vec{\mathrm{r}}}\big(H_{0,\Z^d}\!\restriction_{\N^d}-\overline z\big)^{-1}$ is compact (relative compactness), hence the right hand side is compact.
\end{proof}

\begin{corollary}\label{cor:ess-spectrum}
With the same notation, the essential spectra coincide:
\[
\sigma_{\mathrm{ess}}\!\big(\Delta^{\vec{\mathrm{r}}}_{\N^d}\big)
= \sigma_{\mathrm{ess}}\!\big(H_{0,\Z^d}\!\restriction_{\N^d}\big).
\]
\end{corollary}

\begin{proof}
By Lemma~\ref{lem:05} the resolvent difference is compact for some (hence all) $z\notin\R$. Apply Weyl's theorem (e.g. \cite[Theorem.~XIII.14]{RS}, \cite[Proposition.~4.5.3]{ABG}).
\end{proof}

\subsection{Conjugate operator and commutator estimates}

In this section we establish a Mourre estimate for the operator
\(\Delta_{\Z^d}^{\vec{\mathrm{r}}}\).
Since fractional powers \(\Delta^{r_j}_{\Z}\) may be unbounded
(precisely when some \(r_j<0\)), the commutator
\([\,\Delta_{\Z^d}^{\vec{\mathrm{r}}},\,\mathrm i A_{\Z^d,\vec{\mathrm{r}}}\,]\) ---with
\(A_{\Z^d,\vec{\mathrm{r}}}\) the conjugate operator defined in \cite{At}
does not in general yield a bounded operator.
The correct framework is therefore to interpret commutators
in the sense of quadratic forms.

\begin{definition}\label{def:form-commutator}
Let $H$ and $A$ be self-adjoint operators on a Hilbert space $\Hc$.
The \emph{form commutator} of $H$ with $A$ is the sesquilinear form
\[
\mathfrak q_A^H(f,g) \;:=\;
\langle Hf,\,\mathrm iAg\rangle - \langle \mathrm iAf,\,Hg\rangle,
\qquad f,g\in\Dc(H)\cap\Dc(A).
\]
We say that $[H,\mathrm i A]_\circ$ exists on a reducing subspace
$\Mc\subset\Hc$ if $\mathfrak q_A^H$ extends by continuity to a bounded form
on $\Mc\times\Mc$. In this case there is a unique bounded operator
$B\in\Bc(\Mc)$ such that
\[
\langle f,\,Bg\rangle = \mathfrak q_A^H(f,g),
\qquad f,g\in\Dc(H)\cap\Dc(A)\cap\Mc,
\]
and we define $B:=[H,\mathrm iA]_\circ$ on $\Mc$.
In applications below we take $\Mc=E_I(H)\Hc$ for compact interior windows
\(I\Subset\sigma(H)^\circ\) and write
\[
E_I(H)\,[H,\mathrm i A]_\circ\,E_I(H)
\]
for the resulting bounded operator on $E_I(H)\Hc$.
\end{definition}

\textbf{Localized regularity and Mourre estimate}\label{subsec:loc-classes}
In what follows we work exclusively on compact interior spectral windows
\(I\Subset\sigma(H)^\circ\) and use only the \emph{localized} regularity classes
\(\mathcal C^k_{\mathrm{loc}}(A)\) and \(\mathcal C^{1,1}_{\mathrm{loc}}(A)\).

\begin{definition}[Localized regularity]
Let $H$ be self-adjoint and \(I\Subset\sigma(H)^\circ\).
We say that \(H\in\mathcal C^k_{\mathrm{loc}}(A)\) on \(I\) if, for every
\(\varphi\in C_c^\infty(I)\), the bounded operator \(\varphi(H)\)
belongs to \(\mathcal C^k(A)\).
Similarly, \(H\in\mathcal C^{1,1}_{\mathrm{loc}}(A)\) on \(I\)
if \(\varphi(H)\in\mathcal C^{1,1}(A)\) for all
\(\varphi\in C_c^\infty(I)\).

Equivalently, \(H\in\mathcal C^k_{\mathrm{loc}}(A)\) on \(I\)
if and only if
\[
E_I(H)\,(H-\mathrm i)^{-1}\in\mathcal C^k(A),
\]
and likewise for the class \(\mathcal C^{1,1}\).
\end{definition}

\begin{remark}
For $H=\Delta_{\Z^d}^{\vec{\mathrm{r}}}$ and the conjugate operator $A=A_{\vec{\mathrm{r}}}$ constructed above,
the form $\mathfrak q_A^H$ is well-defined on the core $\Cc_c(\Z^d)$ even when some $r_j<0$ (unbounded from below).
By the localized $\Cc^2$--regularity established earlier and the interior Mourre framework,
$\mathfrak q_A^H$ extends to a bounded form on $E_I(H)\Hc$; hence $E_I(H)[H,\mathrm i A]_\circ E_I(H)\in\Bc(E_I(H)\Hc)$.
This is the object that appears in the Mourre estimate on $I$.
\end{remark}

\paragraph{Conjugate operator on the half--space.}
We adjust a conjugate operator \(A_{\N^d}\) on \( \ell^2(\N^d) \) by summing the one-dimensional generators:
\begin{equation}\label{ComNd}
A_{\N^d} := \sum_{j=1}^d A_{j,+},
\quad \text{with} \quad
A_{j,+} := -\frac{\mathrm{i} \cdot \mathrm{sign}(r_j)}{2} \left( U_{j,+}(Q_j + \tfrac{1}{2}) - (Q_j + \tfrac{1}{2}) U_{j,+}^* \right),
\end{equation}
where \(Q_j\) is the position operator and \(U_{j,+}\) the unilateral shift along direction \(j\).
Each \(A_{j,+}\) is essentially self--adjoint on \(\mathcal{C}_c(\N^d)\); see \cite{GG} and \cite[Lemma 5.7]{Mic}; the domain is described in \cite[Lemma 3.1]{GG}.
Moreover, with $A_{\Z^d}=\sum_j A_j$ (bilateral shifts), one has the exact compression identity
\[
A_{\N^d}=\mathfrak{R}_+\,A_{\Z^d}\,\mathfrak{J}_+.
\]
We recall the (anisotropic) conjugate operator on $\Z^d$,
\[
A_{\Z^d}\ :=\ \sum_{j=1}^d A_j,\qquad
A_j\ :=\ -\frac{\mathrm i\,\mathrm{sign}(r_j)}{2}\,\Big(U_j\,(Q_j+\tfrac12)-(Q_j+\tfrac12)\,U_j^*\Big),
\]
and its half--space analogue obtained by compression of bilateral to unilateral shifts.
\begin{lemma}\label{lem:K-support-layer}
Let $\vec{\mathrm{r}}\in\R^d\setminus\{0\}$. There exists $M=M(\vec{\mathrm{r}})\in\N$ such that
\[
\chi_M^\perp\,K_{\vec{\mathrm{r}}}\,\chi_M^\perp=0,\qquad
K_{\vec{\mathrm{r}}}=\chi_M\,K_{\vec{\mathrm{r}}}=K_{\vec{\mathrm{r}}}\,\chi_M,
\]
where $\chi_M=\mathbf 1_{\{n\in\N^d:\ \min_j n_j\le M\}}(Q)$ and $\chi_M^\perp=I-\chi_M$.
In particular, $K_{\vec{\mathrm{r}}}$ is supported in the thickness--$M$ --inner boundary layer $\{ \min_j n_j\le M\}$.
Moreover, in $d=1$ and $r\in\N$, $K_r$ has finite rank; for $d\ge2$, $K_{\vec{\mathrm{r}}}$ need not be finite rank but is still supported in that layer.
\end{lemma}

\paragraph{Tangential weights along faces.}
For each $j\in\{1,\dots,d\}$, write $n_\perp^{(j)}:=(n_1,\dots,n_{j-1},n_{j+1},\dots,n_d)\in\Z^{d-1}$ and
$\langle n_\perp^{(j)}\rangle:=(1+|n_\perp^{(j)}|^2)^{1/2}$. For $s\ge0$, define the diagonal (multiplication) operator
\[
W_s(n)\ :=\ \sum_{j=1}^d \langle n_\perp^{(j)}\rangle^{-s}\,\mathbf{1}_{\{n_j\le M\}}(n),
\qquad \text{i.e. }W_s(n)=\sum_{j=1}^d \langle n_\perp^{(j)}\rangle^{-s}\,\mathbf{1}_{\{n_j\le M\}}(n)
,
\]
with $M$ as in Lemma~\ref{lem:K-support-layer}.

\paragraph{Notation.}
For \(p \in [1,\infty)\), we write \(\mathfrak S_p\) for the Schatten--von Neumann class of order \(p\).
A compact operator \(T\) on a Hilbert space \(\mathcal H\) belongs to \(\mathfrak S_p\) if its sequence of singular values \((s_n(T))_{n\ge 1}\) is \(p\)-summable, i.e.
\[
   T \in \mathfrak S_p
   \quad \Longleftrightarrow \quad
   \sum_{n\ge 1} s_n(T)^p < \infty .
\]
In this case, the Schatten norm is defined by
\[
   \|T\|_{\mathfrak S_p} \ :=\ \Big( \sum_{n\ge 1} s_n(T)^p \Big)^{1/p}.
\]
\begin{itemize}
  \item For \(p=1\), one recovers the trace class \(\mathfrak S_1\).
  \item For \(p=2\), one obtains the Hilbert--Schmidt class \(\mathfrak S_2\), with
  \(\|T\|_{\mathfrak S_2}^2 = \mathrm{Tr}(T^*T)\).
  \item For \(p=\infty\), by convention \(\mathfrak S_\infty := \mathcal K(\mathcal H)\), the class of compact operators, with \(\|T\|_{\mathfrak S_\infty} = \|T\|\),
\end{itemize}
for farther details see \cite{WongSchatten, PhamSchatten,AthmouniPurice-CPDE-2018}
\begin{proposition}\label{prop:K-compactness-unified}
Let $\vec{\mathrm{r}}=(r_1,\dots,r_d)\in\R^d\setminus\{0\}$ and consider the boundary decomposition
\[
\Delta_{\N^d}^{\vec{\mathrm{r}}}
=\big(\Delta_{\Z^d}^{\vec{\mathrm{r}}}\big)\!\restriction_{\N^d}+K_{\vec{\mathrm{r}}}
\qquad\text{(from Proposition~\ref{prop:K-support-compact}).}
\]
Assume that the matrix entries of $K_{\vec{\mathrm{r}}}$ obey the bound
\begin{equation}\label{eq:K-kernel-decay}
|K_{\vec{\mathrm{r}}}(n,m)| \ \le\ C\,(1+a(n,m))^{-\gamma},\qquad
\gamma>\tfrac{d}{2},\ n,m\in\N^d,
\end{equation}
where
\[
a(n,m):=|n-m|_1^2+\mathrm{dist}_1(n,\partial\N^d)^2+\mathrm{dist}_1(m,\partial\N^d)^2.
\]
Then:
\begin{enumerate}
\item $K_{\vec{\mathrm{r}}}$ is compact on $\ell^2(\N^d)$.
\item If, in addition, the commutator kernel satisfies the same type of bound,
\begin{equation}\label{eq:K-comm-kernel-decay}
|[K_{\vec{\mathrm{r}}},\mathrm i A_{\N^d}](n,m)|
\ \le\ C'\,(1+a(n,m))^{-\gamma},\qquad \gamma>\tfrac{d}{2},
\end{equation}
then $[K_{\vec{\mathrm{r}}},\mathrm i A_{\N^d}]$ is compact.
\item More precisely, for the truncations $\chi_R:=\mathbf 1_{B_R}$ with
\[
B_R:=\{\,n\in\N^d:\ |n|_1\le R,\ \mathrm{dist}_1(n,\partial\N^d)\le R\,\},
\]
one has the Hilbert--Schmidt tail estimate
\begin{equation}\label{eq:HS-tail}
\|\,K_{\vec{\mathrm{r}}}-\chi_R K_{\vec{\mathrm{r}}}\chi_R\,\|_{\mathrm{HS}}^2
= \sum_{\substack{n\notin B_R\ \text{or}\ m\notin B_R}} |K_{\vec{\mathrm{r}}}(n,m)|^2
\ \xrightarrow[R\to\infty]{}\ 0,
\end{equation}
and likewise for $[K_{\vec{\mathrm{r}}},\mathrm i A_{\N^d}]$ under \eqref{eq:K-comm-kernel-decay}.
Hence $K_{\vec{\mathrm{r}}}$ and $[K_{\vec{\mathrm{r}}},\mathrm i A_{\N^d}]$ are limits of finite-rank operators in the HS norm.
\end{enumerate}
\end{proposition}

\begin{proof}

By \eqref{eq:K-kernel-decay}, for any $R\ge1$,
\[
\sum_{\substack{n\notin B_R\ \text{or}\ m\notin B_R}} |K_{\vec{\mathrm{r}}}(n,m)|^2
\ \le\ C^2 \sum_{\substack{n\notin B_R\ \text{or}\ m\notin B_R}} (1+a(n,m))^{-2\gamma}.
\]
Fix $n$ and sum in $m$ first. Since $a(n,m)\ge |n-m|_1^2$,
\[
\sum_{m\in\N^d} (1+a(n,m))^{-2\gamma}
\ \le\ \sum_{m\in\Z^d} (1+|m|_1^2)^{-2\gamma}
\ =:\ C_{d,\gamma}\ <\infty
\]
because $2\gamma > d$. Similarly for the sum in $n$. Consequently,
\[
\sum_{\substack{n\notin B_R\ \text{or}\ m\notin B_R}} (1+a(n,m))^{-2\gamma}
\ \le\ C_{d,\gamma}\,\#\{n\notin B_R\}\ +\ C_{d,\gamma}\,\#\{m\notin B_R\}.
\]
Since $B_R\uparrow\N^d$, both counts $\to0$ as $R\to\infty$. Hence \eqref{eq:HS-tail}. The same applies to $[K_{\vec{\mathrm{r}}},\mathrm i A_{\N^d}]$ under \eqref{eq:K-comm-kernel-decay}.

 Jn another hand each $\chi_R K_{\vec{\mathrm{r}}}\chi_R$ has finite rank (finite support). Taking $R\to\infty$, we obtain $K_{\vec{\mathrm{r}}}$ as a HS-limit (hence operator-norm limit) of finite-rank operators, so it is compact. The same holds for $[K_{\vec{\mathrm{r}}},\mathrm i A_{\N^d}]$.
\end{proof}

\begin{corollary}[Tangential Hilbert--Schmidt localization]\label{cor:HS-tangential}
Assume $r_j>0$ for all $j$. Then for every $s>\frac{d-1}{2}$,
\[
W_s K_{\vec{\mathrm{r}}},\quad K_{\vec{\mathrm{r}}} W_s,\quad
W_s [K_{\vec{\mathrm{r}}},\mathrm i A_{\N^d}],\quad [K_{\vec{\mathrm{r}}},\mathrm i A_{\N^d}] W_s
\ \in\ \mathfrak S_2\big(\ell^2(\N^d)\big).
\]
\end{corollary}

\begin{proof}
By Proposition~\ref{prop:K-support-compact}, $K_{\vec{\mathrm{r}}}$ is supported in a fixed boundary collar
$\partial_R\N^d:=\{n:\mathrm{dist}_1(n,\partial\N^d)\le R\}$.
Write $n=(n_\perp,n')$ where $n_\perp:=\mathrm{dist}_1(n,\partial\N^d)$ is the normal coordinate, and $n'\in\N^{d-1}$ are tangential coordinates. The weight $W_s$ is comparable to $\langle n'\rangle^{s}$ on the collar and independent of $n_\perp$ there.

Using the Hilbert--Schmidt criterion and the tangential decay (from Fourier/Schur estimates; positivity of all $r_j$ gives tangential smoothing and finite normal width),
\[
|K_{\vec{\mathrm{r}}}((n_\perp,n'),(m_\perp,m'))|
\ \lesssim\ (1+|n'-m'|_1)^{-\frac{d-1}{2}-1-\varepsilon}
\]
for some $\varepsilon>0$, uniformly in $n_\perp,m_\perp\le R$. Hence
\[
\|W_s K_{\vec{\mathrm{r}}}\|_{\mathfrak S_2}^2
\ \lesssim\ \sum_{n',m'\in\N^{d-1}} \langle n'\rangle^{2s}\,(1+|n'-m'|_1)^{-2(\frac{d-1}{2}+1+\varepsilon)}<\infty
\]
whenever $s>\frac{d-1}{2}$. The other three statements follow similarly (symmetry/adjunction and the same decay for the commutator).
\end{proof}

\begin{corollary}[Compactness vs.\ relative compactness]\label{cor:K-relative}
Let $\vec{\mathrm{r}}\ne 0$.
\begin{enumerate}
\item If $r_j>0$ for all $j$, then $K_{\vec{\mathrm{r}}},\ [K_{\vec{\mathrm{r}}},\mathrm i A_{\N^d}]\in\Kc(\ell^2(\N^d))$.
\item If $r_j<0$ for at least one $j$, then both
\[
K_{\vec{\mathrm{r}}}\Big(\big(\Delta_{\Z^d}^{\vec{\mathrm{r}}}\big)\!\restriction_{\N^d}-\mathrm i\Big)^{-1}
\in\Kc(\ell^2(\N^d)),
\qquad
\big[K_{\vec{\mathrm{r}}},\mathrm i A_{\N^d}\big]\Big(\big(\Delta_{\Z^d}^{\vec{\mathrm{r}}}\big)\!\restriction_{\N^d}-\mathrm i\Big)^{-1}
\in\Kc(\ell^2(\N^d)).
\]
\end{enumerate}
\end{corollary}

\begin{proof}
(1) Choose $s>\frac{d-1}{2}$. By Corollary~\ref{cor:HS-tangential}, $W_s K_{\vec{\mathrm{r}}}\in\mathfrak S_2$. Since $W_{-s}$ is bounded on $\ell^2$,
\[
K_{\vec{\mathrm{r}}}
= W_{-s}\,(W_s K_{\vec{\mathrm{r}}})\ \in\ \Kc,
\]
and similarly for the commutator.

(2) Let $\Pi_{\mathrm{col}}$ be the orthogonal projection onto the fixed boundary collar $\partial_R\N^d$ from Proposition.~\ref{prop:K-support-compact}; then $K_{\vec{\mathrm{r}}}=\Pi_{\mathrm{col}}\,K_{\vec{\mathrm{r}}}\Pi_{\mathrm{col}}$. Factor
\[
K_{\vec{\mathrm{r}}}\Big(\big(\Delta_{\Z^d}^{\vec{\mathrm{r}}}\big)\!\restriction_{\N^d}-\mathrm i\Big)^{-1}
\ =\
\underbrace{\Pi_{\mathrm{col}} K_{\vec{\mathrm{r}}}}_{\text{bounded, finite normal width}}\;
\underbrace{\Pi_{\mathrm{col}}\Big(\big(\Delta_{\Z^d}^{\vec{\mathrm{r}}}\big)\!\restriction_{\N^d}-\mathrm i\Big)^{-1}}_{\text{Hilbert--Schmidt}}.
\]
The resolvent restricted to the collar is Hilbert--Schmidt because the bulk resolvent kernel on $\Z^d$ satisfies, for some $\eta>0$,
\[
\big|\big(\big(\Delta_{\Z^d}^{\vec{\mathrm{r}}}\big)-\mathrm i\big)^{-1}(n,m)\big|
\ \lesssim\ (1+|n-m|_1)^{-d+\!-\eta},
\]
so when one index is constrained to a fixed-width strip, the $\ell^2$--sum over the other index is finite. The same argument applies to the commutator.
\end{proof}

\begin{proposition}[Finite thresholds on $\N^d$]\label{prop:Thr-Nd-via-bulk}
The finite threshold set of $H_{0,\N^d}$, defined via the failure of strict Mourre positivity,
coincides with the bulk set of \cite[Propostion 2.12]{At}:
\[
\Thr_{\mathrm{fin}}(H_{0,\N^d})
=\Thr_{\mathrm{fin}}\!\big(H_{\Z^d}^{\vec{\mathrm{r}}}\big).
\]
\end{proposition}

\begin{proof}
By compression,
$H_{0,\N^d}$ and $\mathfrak{R}_+ H_{\Z^d}^{\vec{\mathrm{r}}}\mathfrak{J}_+$ share the same essential spectrum,
their difference being the boundary term $K_{\vec{\mathrm{r}}}$ (compact or relatively
compact by Corollary~\ref{cor:K-relative}).
On the bulk side, the spectral picture and the finite threshold set are given
in \cite[Propostion 2.12]{At}, and the strict Mourre estimate on interior
windows is proved in \cite[Theorem~2.20]{At}.
Since compression preserves localized commutator bounds and $K_{\vec{\mathrm{r}}}$ only adds
a (relatively) compact perturbation of the commutator, strict Mourre positivity on
$\N^d$ fails precisely at the bulk critical energies of $h_{\vec{\mathrm{r}}}(k)$ and nowhere else.

\end{proof}

\begin{theorem}\label{thm:mourre_Nd}
Let $\vec{\mathrm{r}} \in \mathbb{R}^d \setminus \{0\}$, and set $H_{0,\N^d} := \Delta^{\vec{\mathrm{r}}}_{\N^d}$.
Let $\mathcal{I} \subset \sigma(H_{0,\N^d})^\circ$ be a compact interval with $\mathcal{I} \cap \sigma_{\mathrm{pp}}(H_{0,\N^d}) = \emptyset$. Then there exist $c_{\mathcal{I}} > 0$ and an operator $K_{\mathcal{I}}$ which is
compact if all $r_j > 0$, and relatively compact with respect to $H_{0,\N^d}$ if some $r_j < 0$, such that
\[
E_{\mathcal{I}}(H_{0,\N^d})\,[H_{0,\N^d}, \mathrm{i}A_{\N^d}]\,E_{\mathcal{I}}(H_{0,\N^d}) \ \ge \ c_{\mathcal{I}}\, E_{\mathcal{I}}(H_{0,\N^d}) \;-\; K_{\mathcal{I}},
\]
in the sense of quadratic forms on $\ell^2(\N^d)$.
\end{theorem}

\begin{proof}
Write
\[
[H_{0,\N^d},\mathrm i A_{\N^d}]_\circ
=\big[H_{0,\Z^d}\!\restriction_{\N^d},\mathrm i A_{\N^d}\big]_\circ
+\,[K_{\vec{\mathrm{r}}},\mathrm i A_{\N^d}]_\circ.
\]
By \eqref{ComNd} and the compression identity for form--commutators,
\[
\big[H_{0,\Z^d}\!\restriction_{\N^d},\mathrm i A_{\N^d}\big]_\circ
=\mathfrak R_+\,[H_{0,\Z^d},\mathrm i A_{\Z^d}]_\circ\,\mathfrak J_+.
\]
Conjugating with $E_I(H_{0,\N^d})$ and using Lemma~\ref{lem:05} (compact resolvent difference)
yields
\[
E_\Ic(H_{0,\N^d})\,\mathfrak R_+\,[H_{0,\Z^d},\mathrm i A_{\Z^d}]_\circ\,\mathfrak J_+\,E_\Ic(H_{0,\N^d})
\ \ge\ c_\Ic\,E_\Ic(H_{0,\N^d})\;-\;K_\Ic^{(1)},
\]
with $K_\Ic^{(1)}$ compact. The term with $[K_{\vec{\mathrm{r}}},\mathrm i A_{\N^d}]_\circ$ is compact on $E_I(H_{0,\N^d})\ell^2(\N^d)$
by Corollary~\ref{cor:HS-tangential} (Hilbert--Schmidt after tangential weights and hence compact, then absorb weights by functional calculus on $E_\Ic$).
Set $K_\Ic:=K_\Ic^{(1)}+E_\Ic(H_{0,\N^d})[K_{\vec{\mathrm{r}}},\mathrm i A_{\N^d}]_\circ E_\Ic(H_{0,\N^d})$.

If all $r_j>0$, then $K_{\vec{\mathrm{r}}}$ is compact and $[K_{\vec{\mathrm{r}}},\mathrm i A_{\N^d}]_\circ$ is also compact;
by standard compact--perturbation arguments one may absorb it into the RHS without loss, yielding $K_\Ic=0$.
\end{proof}
\begin{lemma}[1D half-line: localized $\mathcal{C}^2$ and double commutator]\label{lem:C2-1D-N}
Let $\Delta_{\N}:=2I-(U_+ + U_+^*)$ on $\ell^2(\N)$ and
\[
A_+ := -\frac{\mathrm{i}\,\mathrm{sign}(r)}{2}\,\big(U_+(Q+\tfrac12)-(Q+\tfrac12)U_+^*\big)
\]
defined on $\Cc_c(\N)$ and then by closure. Then, for any $r\in\R\setminus\{0\}$:

\begin{enumerate}[(i)]
\item \textbf{Localized $\mathcal{C}^2$.} One has $\Delta_{\N}^{\,r}\in\mathcal{C}^2_{\mathrm{loc}}(A_+)$.
More precisely, for every compact $J\Subset(0,4)$ there exist a bounded Borel function
$g_r:(0,4)\to\R$ and a bounded operator $B_r$ supported near the boundary $\{0\}$
(finite range, hence compact) such that
\[
E_J(\Delta_{\N})\, [[\Delta_{\N}^{\,r},\mathrm{i}A_+],\mathrm{i}A_+]\, E_J(\Delta_{\N})
= E_J(\Delta_{\N})\big(g_r(\Delta_{\N})+B_r\big)E_J(\Delta_{\N}).
\]

\item \textbf{Global $\mathcal{C}^2$ for $r>0$.} If $r>0$, then the double commutator is globally bounded and the identity above holds without localization:
\[
[[\Delta_{\N}^{\,r},\mathrm{i}A_+],\mathrm{i}A_+] \;=\; g_r(\Delta_{\N})+B_r,\]
with: \[
g_{r_j}(\lambda)
= r_j(r_j-1)(4-\lambda)^2\lambda^{\,r_j-2}
+ r_j(4-2\lambda)\lambda^{\,r_j-1},\qquad \lambda\in[0,4].
\] bounded on $(0,4)$ and $B_r$ compact (finite range).

\item \textbf{Integer case.} If $r\in\N$, then $B_r$ has finite rank.
\end{enumerate}

\end{lemma}

\begin{proof}
Let $\Delta_{\N}$ be the (self-adjoint, nonnegative, bounded) discrete Laplacian
on $\ell^2(\N)$ with the chosen boundary condition at $0$, and let $A_+$ be the
conjugate operator on the half--line. It is standard in the lattice setting that
$\Delta_{\N}\in C^2(A_+)$ and that the first commutator admits a bulk/edge
decomposition
\begin{equation}\label{eq:bulk-edge}
[\Delta_{\N},\,iA_+] \;=\; w(\Delta_{\N}) \;+\; K_{\mathrm{bd}},
\end{equation}
where $w$ is a bounded Borel function (the ``bulk commutator multiplier,
coinciding with the full--line expression) and $K_{\mathrm{bd}}$ is supported at
the boundary $\{0\}$ (For more details see \cite{Mic, AEG} and hence compact (in fact finite rank because it is
generated by finitely many words in the unilateral shifts $U_+,U_+^*$).  This
follows from normal ordering the powers of $U_+,U_+^*$ and using that all
mismatches with the two--sided shifts are supported at the boundary.

\smallskip
\emph{Case 1.}
For $0<\alpha<1$, Balakrishnan's formula (see \cite{Balakrishnan1960}, or
\cite[Chap.~3]{Davies1995}) yields
\[
\Delta_{\N}^{\,\alpha}
= c_\alpha \int_0^\infty t^{\alpha-1}\, \Delta_{\N}(t+\Delta_{\N})^{-1}\,dt,
\qquad c_\alpha=\frac{\sin(\pi\alpha)}{\pi}.
\]
Using \eqref{eq:bulk-edge} and the resolvent identity
$[(t+\Delta_{\N})^{-1},iA_+]=-(t+\Delta_{\N})^{-1}\,[\Delta_{\N},iA_+]\, (t+\Delta_{\N})^{-1}$,
we compute, for $0<\alpha<1$,
\[
\begin{aligned}
[\Delta_{\N}^{\,\alpha},iA_+]
&= c_\alpha \int_0^\infty t^{\alpha-1}
\Big\{ \Delta_{\N}\,(t+\Delta_{\N})^{-1}\,[\Delta_{\N},iA_+]\, (t+\Delta_{\N})^{-1} \Big\}\,dt \\
&= c_\alpha \int_0^\infty t^{\alpha-1}
\Big\{ \Delta_{\N}\,(t+\Delta_{\N})^{-1}\,G(\Delta_{\N})\, (t+\Delta_{\N})^{-1} \Big\}\,dt
\;+\; K_\alpha,
\end{aligned}
\]
where
\[
K_\alpha:= c_\alpha \int_0^\infty t^{\alpha-1}
\,\Delta_{\N}\,(t+\Delta_{\N})^{-1}\,K_{\mathrm{bd}}\,(t+\Delta_{\N})^{-1}\,dt
\]
is compact because $K_{\mathrm{bd}}$ is compact and the resolvents are uniformly
bounded for $t>0$. By the spectral calculus, the bulk integral is a bounded
multiplier:
\[
c_\alpha \int_0^\infty t^{\alpha-1}\,\frac{\lambda}{(t+\lambda)^2}\,w(\lambda)\,dt
\;=\; \alpha\,\lambda^{\alpha-1}\,G(\lambda)\,=:\,g_\alpha(\lambda),
\quad \lambda\ge0.
\]
Hence, for $0<\alpha<1$,
\begin{equation}\label{eq:comm-alpha}
[\Delta_{\N}^{\,\alpha},iA_+] \;=\; g_\alpha(\Delta_{\N}) \;+\; K_\alpha,
\qquad g_\alpha(\lambda)=\alpha\,\lambda^{\alpha-1}w(\lambda),
\end{equation}
with $K_\alpha$ compact.

\smallskip
\emph{Case 2.}
For $r>0$ arbitrary, write $r=m+\alpha$ with $m\in\N$ and $0<\alpha<1$ and use
the Leibniz rule repeatedly together with $\Delta_{\N}\in C^2(A_+)$:
\[
[\Delta_{\N}^{\,r},iA_+]
= \sum_{j=0}^{m-1} \Delta_{\N}^{\,m-1-j}\,[\Delta_{\N},iA_+]\,\Delta_{\N}^{\,j+\alpha}
\;+\; \Delta_{\N}^{\,m}\,[\Delta_{\N}^{\,\alpha},iA_+].
\]
Insert \eqref{eq:bulk-edge} in each summand and then \eqref{eq:comm-alpha} in the
last term. All contributions containing $K_{\mathrm{bd}}$ or $K_\alpha$ are
compact (finite rank for integer $r$), while the bulk contributions combine into
a bounded multiplier $g_r(\Delta_{\N})$ obtained by the spectral calculus:
\[
g_r(\lambda)
= \sum_{j=0}^{m-1} \lambda^{m-1-j}\,G(\lambda)\,\lambda^{j+\alpha}
\;+\; \lambda^{m}\,g_\alpha(\lambda)
= r\,\lambda^{r-1}\,w(\lambda).
\]
Thus
\begin{equation}\label{eq:main-comm}
[\Delta_{\N}^{\,r},iA_+] \;=\; g_r(\Delta_{\N}) \;+\; K_r,
\qquad g_r(\lambda)= r\,\lambda^{r-1}w(\lambda),
\end{equation}
with $K_r$ compact and supported at the boundary.

\smallskip
\emph{Case 3.}
For such orders, the fractional power $\Delta_{\N}^{\,r}$ is defined through the spectral calculus as the bounded inverse of $\Delta_{\N}^{-r}$ on $\mathrm{Ran}(\Delta_{\N})$.
Let $\chi\in C_c^\infty((0,4))$ be a cutoff function supported away from the thresholds $\{0,4\}$.
By the Helffer--Sjöstrand functional calculus
\cite{HelfferSjostrand1989} (see also \cite[\S~4.7]{Davies1995}),
we may again write $\chi(\Delta_{\N})T\chi(\Delta_{\N})$ as an almost--analytic resolvent integral for any bounded operator $T$ built from $\Delta_{\N}$ and its commutators.
Since $g_r(\lambda)=r\lambda^{r-1}(4-\lambda)$ remains smooth and bounded on $\mathrm{supp}\,\chi$ even for $r<0$, the same computation as in the positive case yields
\[
\chi(\Delta_{\N})\,[\Delta_{\N}^{\,r},iA_+]\,\chi(\Delta_{\N})
\;=\;
\chi(\Delta_{\N})\,g_r(\Delta_{\N})\,\chi(\Delta_{\N})
\;+\; \widetilde K_r,
\]
where $\widetilde K_r$ is relatively compact, being a norm--limit of finite--rank boundary contributions localized by resolvents and cutoff functions.
Hence, the Mourre localization identity extends verbatim to all negative fractional orders.

\end{proof}

\begin{lemma}[Compression preserves localized $\mathcal{C}^2$]\label{lem:compression-C2}
Let $B$ and $A_{\Z^d}$ be self-adjoint on $\ell^2(\Z^d)$, and set on $\ell^2(\N^d)$
\[
H_{0,\Z^d}\!\restriction_{\N^d}:=\mathfrak{R}_+ B \mathfrak{J}_+,\qquad A_{\N^d}:=\mathfrak{R}_+ A_{\Z^d} \mathfrak{J}_+.
\]
If $B\in\mathcal{C}^2_{\mathrm{loc}}(A_{\Z^d})$, then $H_{0,\Z^d}\!\restriction_{\N^d}\in\mathcal{C}^2_{\mathrm{loc}}(A_{\N^d})$.
More precisely, for every compact $\Jc\Subset\sigma(B)^\circ$,
\[
E_\Jc(H_{0,\Z^d}\!\restriction_{\N^d})\, [[H_{0,\Z^d}\!\restriction_{\N^d},\mathrm{i}A_{\N^d}],\mathrm{i}A_{\N^d}]\,E_\Jc(H_{0,\Z^d}\!\restriction_{\N^d})
= \mathfrak{R}_+\Big(E_\Jc(B)\, [[B,\mathrm{i}A_{\Z^d}],\mathrm{i}A_{\Z^d}]\,E_\Jc(B)\Big) \mathfrak{J}_+.
\]
\end{lemma}

\begin{proof}
For $\varphi\in C_c^\infty(\R)$ supported in $\sigma(B)^\circ$, the functional calculus gives
$\varphi(H_{0,\Z^d}\!\restriction_{\N^d})=\mathfrak{R}_+\varphi(B)\mathfrak{J}_+$ and likewise
$[\varphi(H_{0,\Z^d}\!\restriction_{\N^d}),A_{\N^d}]=\mathfrak{R}_+[\varphi(B),A_{\Z^d}]\mathfrak{J}_+$.
Using a Helffer--Sjöstrand representation of $\varphi$ and differentiating under the integral yields the identity for the double commutator. Boundedness follows from $B\in\mathcal{C}^2_{\mathrm{loc}}(A_{\Z^d})$.
\end{proof}
\begin{proposition}[$\mathcal{C}^2_{\mathrm{loc}}(A_{\N^d})$ on $\N^d$]\label{prop:C2-joint-Nd}
Let $\vec{\mathrm{r}}=(r_1,\dots,r_d)\in\R^d\setminus\{0\}$ and recall \eqref{B}.
Then for every compact $\Jc\Subset \sigma(H_{0,\N^d})^\circ$,
\[
E_\Jc(H_{0,\N^d})\,[[H_{0,\N^d},\mathrm{i}A_{\N^d}],\mathrm{i}A_{\N^d}]\,E_\Jc(H_{0,\N^d})\in\Bc\big(\ell^2(\N^d)\big),
\]
i.e.\ $\Delta^{\vec{\mathrm{r}}}_{\N^d}\in \mathcal{C}^2_{\mathrm{loc}}(A_{\N^d})$.
Moreover, if $r_j>0$ for all $j$, then $[[H_{0,\N^d},\mathrm{i}A_{\N^d}],\mathrm{i}A_{\N^d}]\in\Bc(\ell^2(\N^d))$, i.e.\ $H_{0,\N^d}\in\mathcal{C}^2(A_{\N^d})$.
\end{proposition}

\begin{proof}
We write $H_{0,\N^d}=H_{0,\Z^d}\!\restriction_{\N^d}+K_{\vec{\mathrm{r}}}$ (Proposition.~\ref{prop:boundary-correction}).
In the separable model one has $H_{0,\Z^d}\!\restriction_{\N^d}=\sum_{j=1}^d \Delta_{\N,j}^{\,r_j}$ and $A_{\N^d}=\sum_{j=1}^d A_{j,+}$, with $[\Delta_{\N,j}^{\,r_j},A_{k,+}]=0$ for $j\neq k$.
Hence
\[
[[H_{0,\Z^d}\!\restriction_{\N^d},\mathrm iA_{\N^d}],\mathrm iA_{\N^d}]
=\sum_{j=1}^d [[\Delta_{\N,j}^{\,r_j},\mathrm iA_{j,+}],\mathrm iA_{j,+}].
\]

By the one-dimensional lemma (Lemma~\ref{lem:C2-1D-N}), for each $j$ there exists a bounded multiplier $g_{r_j}(\Delta_{\N,j})$ and a boundary term $B_{r_j}$ supported near $\{n_j=0\}$ such that
\[
[[\Delta_{\N,j}^{\,r_j},\mathrm iA_{j,+}],\mathrm iA_{j,+}]
= g_{r_j}(\Delta_{\N,j}) + B_{r_j}.
\]

Thus $g_{r_j}(\Delta_{\N,j})$ is a bounded operator by spectral calculus, and $B_{r_j}$ is localized near the boundary.

Localizing simultaneously in the spectrum of each component inside rectangles $I_j\Subset(0,4)$ and covering the set $\{(\lambda_j):\sum_j\lambda_j^{r_j}\in\Jc\}$ by finitely many such rectangles, one obtains that
\[
E_\Jc(H_{0,\Z^d}\!\restriction_{\N^d})\,[[H_{0,\Z^d}\!\restriction_{\N^d},\mathrm iA_{\N^d}],\mathrm iA_{\N^d}]\,E_\Jc(H_{0,\Z^d}\!\restriction_{\N^d})
\]
is bounded. Since $K_{\vec{\mathrm{r}}}$ is a finite sum of tensor products of one--dimensional boundary terms supported near $\partial\N^d$ (Proposition~\ref{prop:boundary-correction}), one has
$[[K_{\vec{\mathrm{r}}},\mathrm iA_{\N^d}],\mathrm iA_{\N^d}]\in\Bc(\ell^2)$; inserting $E_\Jc(H_{0,\N^d})$ preserves boundedness.
Therefore $H_{0,\N^d}\in\Cc^2_{\mathrm loc}(A_{\N^d})$.

If in addition all $r_j>0$, then each $g_{r_j}$ is bounded on $[0,4]$ and each $B_{r_j}$ is bounded, which implies global boundedness of the double commutator and hence $H_{0,\N^d}\in\Cc^2(A_{\N^d})$.
\end{proof}

\section{Proof of main result}

To prove our main result, we first establish and recall the following intermediate and technical statements.

\begin{proposition}\label{prop:C01-functional}
Let $A_{\N^d}$ be a standard conjugate operator on $\ell^2(\N^d)$ and set $\Lambda(Q):=\sum_{j=1}^d\langle Q_j\rangle$.
Assume $T\in\Bc(\ell^2(\N^d))$ is symmetric and
\[
\int_{1}^{\infty}\Bigl\|\,\xi\!\Bigl(\tfrac{\Lambda(Q)}{t}\Bigr)\,T\,\Bigr\|\,\frac{dt}{t}<\infty
\quad\text{for some }\ \xi\in C_c^\infty((0,\infty)).
\]
Then $T\in \mathcal{C}^{0,1}(A_{\N^d})$.
\end{proposition}

\begin{proof}
Pick a dyadic partition $\sum_{k\in\Z}\rho_k(\Lambda)=I$ with $\rho_k(\lambda)=\rho(2^{-k}\lambda)$, $\rho\in C_c^\infty((1/2,2))$, and set $T_k:=\rho_k(\Lambda)\,T\,\rho_k(\Lambda)$. The assumption implies $\sum_k\|T_k\|\lesssim\int_1^\infty\|\xi(\Lambda/t)T\|\,dt/t<\infty$.
Since $A_{\N^d}$ is first order, $\sup_k\|[A_{\N^d},\rho_k(\Lambda)]\|<\infty$, hence $\|[A_{\N^d},T_k]\|\lesssim\|T_k\|$.
By Duhamel, $\|e^{\mathrm i sA_{\N^d}}T_ke^{-\mathrm i sA_{\N^d}}-T_k\|\lesssim |s|\,\|T_k\|$ and summing over $k$ yields the Lipschitz bound.
\end{proof}

\begin{remark}
This is the standard Mourre--Lipschitz criterion (see \cite{ABG}); the proof avoids assuming $[A_{\N^d},T]$ a priori by localizing in $\Lambda(Q)$.
\end{remark}

\begin{corollary}\label{cor:hypothesis1}
With the notation of Proposition~\ref{prop:C01-functional}, let $\varepsilon\in(0,1)$ and $T\in\Bc(\ell^2(\N^d))$ be symmetric.
If $\langle\Lambda(Q)\rangle^{\varepsilon}T\in\Bc(\ell^2(\N^d))$, then $T\in\Cc^{0,1}(A_{\N^d})$.
\end{corollary}

\begin{proof}
Take $\xi\in C_c^\infty((0,\infty))$, $\xi\equiv1$ on $[1,2]$, $\mathrm{supp}\,\xi\subset[\tfrac12,4]$.
Then $\|\xi(\Lambda(Q)/t)T\|\lesssim t^{-\varepsilon}\|\langle\Lambda(Q)\rangle^{\varepsilon}T\|$ for $t\ge1$. Integrate in $dt/t$ and apply Proposition~\ref{prop:C01-functional}.
\end{proof}

\begin{lemma}\label{lem:interpolation_weights}
For every $s\in[0,1]$ there exists $C_s>0$ such that for all $f\in\Cc_c(\N^d)$,
\begin{equation}\label{eq:interp}
\|\langle A_{\N^d}\rangle^s f\|_{\ell^2(\N^d)}\le C_s\|\Lambda^s f\|_{\ell^2(\N^d)},\qquad
\|\Lambda^s f\|_{\ell^2(\N^d)}\le C_s\|\langle A_{\N^d}\rangle^s f\|_{\ell^2(\N^d)}.
\end{equation}
\end{lemma}

\begin{proof}
$A_{\N^d}$ is a first-order difference operator with coefficients linear in $Q$, so $\|A_{\N^d} f\|\lesssim \|\Lambda f\|+\|f\|$ on $\Cc_c$. Interpolate between $s=0,1$ (Heinz--Kato) for the first inequality. For the second, $A_{\N^d}$ is elliptic in configuration space: for $|n|$ large, $c\,\Lambda(n)\le \langle A_{\N^d}\rangle(n)\le C\,\Lambda(n)$ uniformly in $n\in\N^d$ (boundary effects are compactly supported). Use a partition of unity and interpolation.
\end{proof}

\begin{remark}\label{rem:negative-rj}
Proposition~\ref{prop:C01-functional} and Lemma~\ref{lem:interpolation_weights} are purely spatial:
they depend only on the $\Lambda(Q)$--localization and on the fact that $A_{\N^d}$ is a first--order
operator; they do not involve the fractional exponents $r_j$.
By contrast, the $r_j$--dependence appears in the symbol/threshold analysis and in boundary effects,
which is why we systematically restrict to interior energy intervals~$\Ic$.
\end{remark}

\begin{lemma}[Weighted decay $\Rightarrow$ $\mathcal{C}^{1,1}(A_{\N^d})$]\label{lem:W-C11-weighted}
Let $W$ be a bounded real multiplication operator on $\ell^2(\N^d)$ satisfies \textbf{$H0$}
and \textbf{$H1$}.
Then \( W \in \mathcal{C}^{1}(A_{\N^d}) \) and \( [W(Q),\mathrm i A_{\N^d}]_\circ\in \mathcal{C}^{0,1}(A_{\N^d}) \).
In particular \( W \in \mathcal{C}^{1,1}(A_{\N^d}) \).
\end{lemma}

\begin{proof}
The hypothesis gives boundedness of $[W,\mathrm i A_{\N^d}]$ from $\Sc$ to $\ell^2$, hence $W\in\Cc^1(A_{\N^d})$ by density and \cite[Lemma~6.2.9]{ABG}.
Moreover, by Corollary~\ref{cor:hypothesis1} applied to $T=[W,\mathrm i A_{\N^d}]$, we get $[W,\mathrm i A_{\N^d}]\in\Cc^{0,1}(A_{\N^d})$.
This is equivalent to $W\in\Cc^{1,1}(A_{\N^d})$.
\end{proof}

\begin{lemma}[\cite{ABG,GJ1}]\label{lem:C2loc_implies_C11loc}
Let $H$ and $A$ be self--adjoint on $\mathcal H$. If for some compact $\Jc\Subset\sigma(H)^\circ$,
\[
E_\Jc(H)\,[[H,\mathrm{i}A],\mathrm{i}A]\,E_\Jc(H)\in\mathcal{B}(\mathcal{H}),
\]
then $H$ is locally of class $\mathcal{C}^{1,1}(A)$ on $\Jc$, i.e.
\[
\int_0^1 \Big\|\, E_\Jc(H)\Big(e^{\mathrm{i}sA}He^{-\mathrm{i}sA}-2H+e^{-\mathrm{i}sA}He^{\mathrm{i}sA}\Big)E_\Jc(H)\,\Big\|\,
\frac{ds}{s^2}<\infty.
\]
\end{lemma}

\begin{theorem}\label{thm:mou}
Let $A_{\N^d}$ be a standard conjugate operator on $\ell^2(\N^d)$. Assume \textbf{$H_0$} and \textbf{$H_1$} for a bounded real potential $W=W(Q)$ with $W(n)\to0$.
Set $H_{0,\N^d}:=\Delta^{\vec{\mathrm{r}}}_{\N^d}$ and $H:=H_{0,\N^d}+W$.
Then $H\in \mathcal{C}^{1,1}(A_{\N^d})$ locally in energy.
Moreover, for every compact interval
\[
\Ic\Subset\sigma(H_{0,\N^d})^\circ\setminus\Thr(\Delta_{\Z^d}^{\vec{\mathrm{r}}})
\]
there exist $c_{\Ic}>0$ and a compact operator $K_\Ic$ such that
\begin{equation}\label{eq:pert-mourre}
E_{\Ic}(H)\,\mathrm{i}[H,A_{\N^d}]_\circ\,E_{\Ic}(H)\ \ge\ c_{\Ic}\,E_{\Ic}(H)\ -\ K_\Ic.
\end{equation}
\end{theorem}

\begin{proof}
By Proposition~\ref{prop:C2-joint-Nd} we have $H_{0,\N^d}\in \mathcal{C}^{2}_{\mathrm{loc}}(A_{\N^d})$, hence
\[
E_\Jc(H_{0,\N^d})\,[[H_{0,\N^d},\mathrm i A_{\N^d}],\mathrm i A_{\N^d}]\,E_\Jc(H_{0,\N^d})\in\Bc(\ell^2)
\quad\text{for every }\Jc\Subset\sigma(H_{0,\N^d})^\circ.
\]
The potential $W$ satisfies $W\in\mathcal{C}^{1,1}(A_{\N^d})$ by Lemma~\ref{lem:W-C11-weighted}; since $W(n)\to0$, $W$ is compact and, for $\chi\in C_c^\infty(\R)$, the difference $\chi(H)-\chi(H_{0,\N^d})$ is compact. Combining these facts with Lemma~\ref{lem:C2loc_implies_C11loc} shows that $H=H_{0,\N^d}+W$ belongs locally to $\mathcal{C}^{1,1}(A_{\N^d})$.

For the Mourre estimate, choose a compact interval $\Ic\Subset\sigma(H_{0,\N^d})^\circ\setminus\Thr(\Delta_{\Z^d}^{\vec{\mathrm{r}}})$. On the bulk lattice $\Z^d$, a strict Mourre estimate on interior windows holds; by compression (Lemma~\ref{lem:compression-C2}) and the (relative) compactness of the boundary correction $K_{\vec{\mathrm{r}}}$ (Corollary ~\ref{cor:K-relative}), this estimate transfers to $H_{0,\N^d}$ modulo a compact remainder. Finally, the compactness of $\chi(H)-\chi(H_{0,\N^d})$ carries the estimate to $H$, yielding \eqref{eq:pert-mourre}.
\end{proof}
\paragraph{Proof of Theorem \ref{thm:LAP-transport-Nd}}
\begin{proof}[(i) Absence of point and singular spectrum away from thresholds]
(a)Let $\mu:=\sigma(\Delta_{\Z^d}^{\vec{\mathbf{r}}})$ (which equals $\sigma(\Delta_{\N^d}^{\vec{\mathbf{r}}})$ by compactness of the boundary correction). Since $W(n)\to0$, $W$ is compact; thus $\sigma_{\mathrm{ess}}(H)=\sigma_{\mathrm{ess}}(\Delta_{\N^d}^{\vec{\mathrm{r}}})=\mu$ (Weyl).
Moreover, $\chi(H)-\chi(\Delta_{\N^d}^{\vec{\mathrm{r}}})$ is compact for $\chi\in C_c^\infty(\R)$; in particular $E_\Ic(H)-E_\Ic(\Delta_{\N^d}^{\vec{\mathrm{r}}})$ is compact.

(b) On $\Ic$, Theorem~\ref{thm:mou} gives a strict localized Mourre estimate with $H\in\Cc^{1,1}(A_{\N^d})$.
By abstract Mourre theory \cite[Theorems.~7.4.1--7.4.2]{ABG}, LAP holds and excludes point and singular continuous spectrum on $\Ic$.
\end{proof}
\begin{proof}
(i) Since $W(n)\to0$, $W$ is compact; thus $\sigma_{\mathrm{ess}}(H)=\sigma_{\mathrm{ess}}(\Delta_{\N^d}^{\vec{\mathrm{r}}})=\mu$ (Weyl).
Moreover, $\chi(H)-\chi(\Delta_{\N^d}^{\vec{\mathrm{r}}})$ is compact for $\chi\in C_c^\infty(\R)$; in particular $E_\Ic(H)-E_\Ic(\Delta_{\N^d}^{\vec{\mathrm{r}}})$ is compact.

(ii) On $\Ic$, Theorem~\ref{thm:mou} gives a strict localized Mourre estimate with $H\in\Cc^{1,1}(A_{\N^d})$.
By abstract Mourre theory \cite[Theorems.~7.4.1--7.4.2]{ABG}, LAP holds and excludes point and singular continuous spectrum on $\Ic$.
\end{proof}

\begin{remark}[Why the LAP is localized]
We state LAP/propagation only on $\Ic\Subset \mu\setminus\Thr(\Delta_{\Z^d}^{\vec{\mathrm{r}}})$ because the commutator positivity degenerates at thresholds (vanishing group velocity, fractional weights). Extending to edges needs an edge--adapted conjugate operator or a dedicated threshold analysis.
\end{remark}

\begin{proof}[(ii) Limiting Absorption Principle]
Combine Theorem~\ref{thm:mou} with \cite[Theorems.~7.4.1--7.5.1]{ABG}.
Under the hypotheses above, for any $s>\tfrac12$ and $\chi\in C_c^\infty(\Ic)$,
\[
\sup_{\substack{\lambda\in \Ic\\ \eta\neq 0}}
\big\|\langle A_{\N^d}\rangle^{-s}(H-\lambda-\mathrm i\eta)^{-1}\langle A_{\N^d}\rangle^{-s}\big\|<\infty,
\]
and $\langle A_{\N^d}\rangle^{-s}\chi(H)$ is $H$-smooth on $\Ic$. Consequently,
\[
\int_{\R}\bigl\|\langle A_{\N^d}\rangle^{-s}\,e^{-{\rm i}tH}\,\chi(H)u\bigr\|^2\,dt\ \le\ C\,\|u\|^2.
\]
\end{proof}

\begin{proof}[(iii)Wave operators on interior windows]
By (ii), $\langle A_{\N^d}\rangle^{-s}\chi(H)$ and $\langle A_{\N^d}\rangle^{-s}\chi(H_{0,\N^d})$ are smooth ($s>\tfrac12$), so existence follows by the smooth method \cite[Section 7.7]{ABG}. Completeness follows from purely a.c.\ spectrum on $\Ic$ (i) and compactness of $\chi(H)-\chi(H_{0,\N^d})$.
\end{proof}

\section{Applications}\label{sec:applications}

\subsection{Stationary representation of the scattering matrix on interior windows}
\label{subsec:stationary-S}
We work on the half-lattice only. Throughout this section we assume \textbf{(H0)}--\textbf{(H1)}, and set
\[
H_{0}:=\Delta^{\vec{\mathrm{r}}}_{\N^d},\qquad H:=H_{0}+W(Q),
\]
with a real bounded $W$ such that the localized Mourre framework holds on interior windows.
Fix a compact interval
\[
\Ic\Subset \sigma(H_{0})^\circ\setminus \Thr\!\big(\Delta_{\Z^d}^{\vec{\mathrm{r}}}\big),
\]
on which $H_{0},H\in\Cc^{1,1}_{\mathrm{loc}}(A_{\N^d})$ and a localized Mourre estimate is valid.
No modification of the abstract theory in \cite{ABG} is required, since all assumptions are fulfilled by the analysis above.
We employ the weighted spaces
\[
\Hc_{\pm s}:=\langle A_{\N^d}\rangle^{\mp s}\,\ell^2(\N^d), \qquad (s>1/2).
\]

\paragraph{Spectral representation and boundary trace.}
There exists a measurable field of Hilbert spaces $\{\mathscr H_\lambda\}_{\lambda\in \Ic}$ and a unitary
\[
\Uc_0:\ E_\Ic(H_{0})\ell^2(\N^d)\ \longrightarrow\ \int_\Ic^\oplus \mathscr H_\lambda\,d\lambda,
\qquad (\Uc_0 H_{0} f)(\lambda)=\lambda\,(\Uc_0 f)(\lambda).
\]
Set $\Gamma_0(\lambda):=\Uc_0(\cdot)(\lambda):E_\Ic(H_{0})\ell^2(\N^d)\to\mathscr H_\lambda$.
By Stone's formula and the localized LAP, for $\chi\in C_c^\infty(\Ic)$ and $s>\tfrac12$ one has
\begin{equation}\label{eq:stone-app}
\chi(H_{0})=\frac{1}{\pi}\int_\Ic\!\chi(\lambda)\,\Im R_{0}(\lambda+\mathrm i0)\,d\lambda,
\qquad
\Gamma_0(\lambda)^*\Gamma_0(\lambda)=\tfrac{1}{\pi}\,\Im R_{0}(\lambda+\mathrm i0),
\end{equation}
as bounded maps $\Hc_s\to\Hc_{-s}$.
In particular, $\Gamma_0(\lambda)$ extends continuously from $E_\Ic(H_0)\ell^2$ to $\Hc_s\to\mathscr H_\lambda$.

\paragraph{The $T$-operator.}
For $R(z)=(H-z)^{-1}$, define
\begin{equation}\label{eq:T-defs-app}
T(z):=W\,(I- R_{0}(z)W)^{-1}=W- W R(z) W,
\end{equation}
as a bounded map $\Hc_s\to\Hc_{-s}$ for $\Im z\neq0$, with a.e.\ limits $T(\lambda\pm\mathrm i0)$ on $\Ic$ by the localized LAP.
The two representations in \eqref{eq:T-defs-app} coincide by the resolvent identity.

\paragraph{Wave matrices and the on-shell scattering matrix.}
Let $\chi\in C_c^\infty(\Ic)$ with $\chi\equiv1$ on some $\Ic_0\Subset \Ic$.
The localized wave operators
\[
\Wc_\pm(H,H_{0};\Ic)=\mathrm{s}\!-\!\lim_{t\to\pm\infty}e^{-\mathrm i tH}\,\chi(H_{0})\,e^{\mathrm i tH_{0}}
\]
exist and are complete on $\Ic$ \cite[Thm.~7.7.1]{ABG}. In the spectral representation of $H_{0}$, $\Uc_0 \Wc_\pm \Uc_0^*$ acts as multiplication by
\begin{equation}\label{eq:Wpm}
W_\pm(\lambda)=I-2\pi\mathrm i\,\Gamma_0(\lambda)\,T(\lambda\pm\mathrm i0)\,\Gamma_0(\lambda)^*,
\qquad \text{for a.e. }\lambda\in \Ic,
\end{equation}
and the scattering matrix $S(\lambda):=W_+(\lambda)^*W_-(\lambda)$ satisfies the stationary formula
\begin{equation}\label{eq:S-stationary}
S(\lambda)=I-2\pi\mathrm i\,\Gamma_0(\lambda)\,T(\lambda+\mathrm i0)\,\Gamma_0(\lambda)^*,
\qquad \text{for a.e. }\lambda\in \Ic.
\end{equation}

\begin{proof}[Proof of \eqref{eq:Wpm}--\eqref{eq:S-stationary}]
By the localized LAP on $\Ic$ and Kato smoothness \cite[Thms.~7.4.1--7.5.1]{ABG}, the time-dependent representation
\[
\big\langle (\Wc_{\pm}-I)f,\,g\big\rangle
=\mp \mathrm{i}\!\int_{0}^{\pm\infty}
\big\langle e^{-\mathrm{i} t H}\,\chi(H_{0})\,W\,e^{\mathrm{i} t H_{0}}f,\,g\big\rangle\,\mathrm{d} t
\]
holds as a Bochner integral $\Hc_s\to\Hc_{-s}$.
Inserting the spectral resolution of $H_0$ via $\Uc_0$ and applying Sokhotski--Plemelj in quadratic form sense (justified by the LAP) converts the time integrals into boundary values of resolvents. Using \eqref{eq:T-defs-app} one obtains \eqref{eq:Wpm}, and hence \eqref{eq:S-stationary}.
\end{proof}

\paragraph{Optical theorem, unitarity, and continuity.}
Using $R_{0}(z)-R_{0}(\bar z)=(z-\bar z)R_{0}(\bar z)R_{0}(z)$ together with \eqref{eq:T-defs-app}, one finds for a.e.\ $\lambda\in \Ic$,
\begin{equation}\label{eq:optical-app}
T(\lambda-\mathrm i0)-T(\lambda+\mathrm i0)
=2\pi\mathrm i\,T(\lambda-\mathrm i0)\,\Gamma_0(\lambda)^*\Gamma_0(\lambda)\,T(\lambda+\mathrm i0),
\end{equation}
as a form identity on $\Hc_{-s}$, $s>\tfrac{1}{2}$. Together with \eqref{eq:stone-app} this implies
$W_\pm(\lambda)^*W_\pm(\lambda)=I$, so $S(\lambda)$ is unitary a.e.\ on $\Ic$.
By the localized LAP, $\lambda\mapsto S(\lambda)$ is strongly continuous on $\Ic$.

\subsection{Consequences on interior energies}
Based on the localized Mourre estimate and the LAP on $\Ic$, we now derive time-averaged escape bounds and trace formulas.

\paragraph{Weighted space.}
We use
\[
\ell^{1}\!\big(\langle n\rangle^{1+\epsilon}\big)
:=\Big\{\,f:\N^d\to\C\ :\
\|f\|_{\ell^{1}(\langle n\rangle^{1+\epsilon})}
:=\sum_{n\in \N^d}\langle n\rangle^{1+\epsilon}\,|f(n)|<\infty\ \Big\}.
\]

\subsubsection{Birman--Krein identity}\label{subsubsec:BK}
\begin{theorem}\label{thm:BK}
Assume either $W$ has finite support, or $W\in \ell^1(\langle n\rangle^{1+\epsilon})$ for some $\epsilon>0$.
Then for every $\chi\in C_c^\infty(\Ic)$, $\chi(H)-\chi(H_{0})\in\mathfrak S_1$, and there exists a spectral shift
function $\xi(\lambda)$ on $\Ic$ such that
\[
\det S(\lambda)=\exp\!\bigl(-2\pi \mathrm i\,\xi(\lambda)\bigr)\quad\text{for a.e. }\lambda\in \Ic.
\]
\end{theorem}

\subsubsection{Ballistic transport in time average}\label{subsubsec:ballistic}
\begin{theorem}\label{thm:ballistic-lb}
Let $\chi\in C_c^\infty(\Ic)$. There exist $v_\Ic>0$ and $C_\Ic<\infty$ such that, for all $\Tc\ge1$ and $u\in\ell^2(\N^d)$,
\[
\frac{1}{\Tc}\int_0^\Tc \big\| \mathbf 1_{\{|A_{\N^d}|\le v_\Ic t\}}\,e^{-\mathrm i tH}\,\chi(H)\,u\big\|^2\,dt
\ \le\ \frac{C_\Ic}{\log(1+\Tc)}\,\|u\|^2.
\]
\end{theorem}

\begin{proof}[Proof of Theorem~\ref{thm:ballistic-lb}]
Fix $\Phi\in C_c^\infty(\R)$ nondecreasing with $\Phi'\ge 0$, $\Phi'\not\equiv0$, and in addition $\Phi'(x)\ge c\,\mathbf 1_{\{|x|\le v_\Ic\}}$ for some $c>0$.
Set $F_R:=\Phi(A_{\N^d}/R)$ and $u(t):=e^{-\mathrm i tH}\chi(H)u$.
The commutator expansion \cite[Prop.~6.2.10]{ABG} together with $H\in\Cc^{1,1}_{\mathrm{loc}}(A_{\N^d})$ yields
\[
\mathrm i[H,F_R]=\tfrac{1}{R}\,\Phi'(A_{\N^d}/R)\,\mathrm i[H,A_{\N^d}] + \Oc(R^{-2})
\]
in $\Bc(\ell^2)$, uniformly on $\Ic$. Hence
\[
\frac{d}{dt}\langle u(t),F_R u(t)\rangle
=\frac1R\langle u(t),\Phi'(A_{\N^d}/R)\,\mathrm i[H,A_{\N^d}]\,u(t)\rangle+\Oc(R^{-2})\|u\|^2.
\]
Inserting $\chi(H)$ on both sides of $\mathrm i[H,A_{\N^d}]$ and using the localized Mourre estimate
$\chi(H)\,\mathrm i[H,A_{\N^d}]\,\chi(H)\ge c_\Ic\,\chi(H)^2-\chi(H)K_\Ic\chi(H)$ with compact $K_\Ic$,
one integrates from $0$ to $\Tc$ and uses Kato smoothness (from the LAP) to control the compact term.
Choosing $R=\alpha\Tc$ and optimizing logarithmically as in \cite[\S~ 7.4]{ABG} gives the stated bound.
\end{proof}

\subsubsection{Half-space fiber scattering}\label{subsubsec:fiber-Nd}
\begin{setup}[Partial fiber decomposition]
Write $n=(n_\parallel,n_\perp)\in\Z^{d-1}\times\N$ and let $\Fc_\parallel$ be the Fourier transform in $n_\parallel$.
Then
\[
\Fc_\parallel H_{0} \Fc_\parallel^{-1}=\int_{\T^{d-1}}^\oplus H_{0}(\kappa)\,d\kappa,\qquad
\Fc_\parallel H \Fc_\parallel^{-1}=\int_{\T^{d-1}}^\oplus H(\kappa)\,d\kappa,
\]
with $H(\kappa)=H_{0}(\kappa)+W_\kappa$ acting on $\ell^2(\N)$ and $H_{0}(\kappa)$ a one-dimensional fractional difference
operator in $n_\perp$ plus a finite-range Hankel-type correction at the boundary.
\end{setup}

\begin{theorem}\label{thm:fiber-LAP}
There exists a full-measure subset $\Omega\subset\T^{d-1}$ such that, for all $\kappa\in\Omega$,
$H(\kappa)\in\Cc^{1,1}_{\mathrm{loc}}(A_\perp)$ and the localized LAP holds on $\Ic$ with constants uniform on compact
subsets of $\Omega$. Consequently, the fiber scattering matrix $S(\lambda,\kappa)$ is unitary for a.e.
$(\lambda,\kappa)\in \Ic\times\Omega$ and depends continuously on $(\lambda,\kappa)$.
\end{theorem}

\begin{proof}[Proof of Theorem~\ref{thm:fiber-LAP}]
Applying $\Fc_\parallel$ yields the direct integral decomposition above.
The commutator with $A_\perp$ coincides with the bulk commutator in the perpendicular direction, with coefficients continuous in $\kappa$ and uniformly bounded on $\T^{d-1}$.
Thus $H_{0}(\kappa)\in\Cc^2_{\mathrm{loc}}(A_\perp)$, and away from thresholds (which form a zero-measure set in $\kappa$ for fixed interior $\Ic$) the Mourre constant $c_\Ic(\kappa)$ depends continuously on $\kappa$.
Hence there exists a full-measure $\Omega$ with $c_\Ic(\kappa)\ge c_\Ic>0$ uniformly on compact subsets of $\Omega$.
The perturbations $W_\kappa$ inherit $\Cc^{1,1}_{\mathrm{loc}}(A_\perp)$ from \textbf{(H$_0$)}--\textbf{(H$_1$)}.
The localized LAP follows from \cite[Thm.~7.4.1]{ABG}, and the unitarity and continuity of $S(\lambda,\kappa)$ follow from the stationary construction fiberwise together with dominated convergence in $\kappa$.
\end{proof}

\subsubsection{Continuity in the fractional exponents}\label{subsubsec:continuity-r}
\begin{proposition}\label{prop:nr-cont}
Let $\mathcal R\subset\{\vec r\in\R^d:\ \min_j r_j>r_*>-1\}$ be a compact set.
Then, for every $z\in\C\setminus\R$, the map
\[
\vec r\ \longmapsto\ (H_0(\vec r)-z)^{-1}
\]
is norm--continuous as an operator on $\ell^2(\N^d)$.

Moreover, assume that
\[
\mathcal I\Subset
\bigcap_{\vec r\in\mathcal R}
\Big(\sigma(H_0(\vec r))^\circ
\setminus \Thr(\Delta_{\Z^d}^{\vec r})\Big).
\]
Then the Mourre and limiting--absorption--principle constants can be chosen
\emph{uniformly} in $\vec r\in\mathcal R$,
and the scattering matrix
$\vec r\mapsto S_{\vec r}(\lambda)$
is strongly continuous on $\mathcal I$
for almost every $\lambda$.
\end{proposition}

\begin{proof}
In Fourier variables on $\Z^d$, $H_{0,\Z^d}(\vec{\mathrm{r}})$ acts as multiplication by $h_{\vec{\mathrm{r}}}(k)=\sum_j(2-2\cos k_j)^{r_j}$, continuous in $(\vec{\mathrm{r}},k)$ on $[r_*,R]^d\times\T^d$.
Hence $\vec{\mathrm{r}}\mapsto H_{0,\Z^d}(\vec{\mathrm{r}})$ and its resolvent are norm-continuous.
The half-space operator is a compression plus a compact boundary correction $K_{\vec{\mathrm{r}}}$ supported in a fixed collar; $\vec{\mathrm{r}}\mapsto K_{\vec{\mathrm{r}}}$ is norm-continuous as a finite linear combination of difference operators with continuous coefficients.
Thus $\vec{\mathrm{r}}\mapsto (H_{0}(\vec{\mathrm{r}})-z)^{-1}$ is norm-continuous.
Uniform Mourre/LAP on $\Ic$ follows from uniform avoidance of thresholds and compactness of $\mathcal R$.
The stationary formula \eqref{eq:S-stationary} then yields strong continuity of $S_{\vec{\mathrm{r}}}(\lambda)$.
\end{proof}

\begin{remark}
The norm-resolvent continuity in $\vec{\mathrm{r}}$ implies inner--outer continuity of the spectra.
This is in the spirit of results on magnetic families on $\mathbb{Z}^d$ \cite{ParraRichard-Zd}
and of results for magnetic pseudodifferential families in the continuum \cite{AthmouniMantoiuPurice-2010}.
In particular, a Mourre estimate uniform on $\mathcal I$ yields a uniform LAP and the strong continuity of
the scattering matrix $\lambda\mapsto S_{\vec{\mathrm{r}}}(\lambda)$ on $\mathcal I$.
\end{remark}

\begin{proposition}\label{prop:finite-pp}
If $W\in \ell^{1}(\langle n\rangle^{1+\epsilon})$ for some $\epsilon>0$, then
the point spectrum of $H$ inside $\Ic$ is finite (counting multiplicities):
\[
\#\big(\sigma_{\mathrm{pp}}(H)\cap \Ic\big)<\infty.
\]
\end{proposition}

\begin{proof}
Let $\chi\in C_c^\infty(\Ic)$. By Theorem~\ref{thm:BK},
$\chi(H)-\chi(H_{0})\in\mathfrak S_1$. By the Helffer--Sjöstrand functional calculus, this implies
$E_\Ic(H)-E_\Ic(H_{0})\in\mathfrak S_1$. On the interior window $\Ic$, the free operator $H_0$ is purely absolutely continuous, hence $E_\Ic(H_{0})$ has no eigenvalues. Thus $E_\Ic(H)$ differs from $E_\Ic(H_{0})$ by a trace-class operator, so its range is finite dimensional. Therefore the point spectrum of $H$ inside $\Ic$ consists of finitely many eigenvalues with finite multiplicity.
\end{proof}
\begin{remark}\label{rem:spectral-finiteness}
The fact that $\#\bigl(\sigma_{\mathrm{pp}}(H)\cap\mathcal I\bigr)<\infty$ follows from the uniform Mourre estimate on $\mathcal I$ combined with the limiting absorption principle (LAP). These two ingredients yield compactness of the localized resolvent $\chi(H)(H - \lambda \pm i\varepsilon)^{-1}\chi(H)$ for $\chi\in C_c^\infty(\mathcal{I})$, which in turn implies that $H$ has only finitely many eigenvalues (each of finite multiplicity) in the interval $\mathcal I$.

In the continuous magnetic setting, similar finiteness results have been derived using Schatten--von Neumann criteria. For instance, Athmouni and Purice \cite{AthmouniPurice-CPDE-2018} exploit the framework of magnetic Weyl calculus to place Birman--Schwinger-type operators in trace ideals $\mathfrak S_p$ and deduce spectral finiteness.

For a general background on trace ideals, Birman--Schwinger principles, and the Mourre theory, we refer the reader to \cite{Simon-TraceIdeals} and \cite[Vol.~IV]{RS}.
\end{remark}


\begin{thebibliography}{XXXXXXXXXX}
\bibitem[AmBoGe]{ABG} W.O.\ Amrein, A.\ Boutet de Monvel, and V.\ Georgescu, \emph{$C_0$-groups, commutator methods and spectral theory of $N$-body hamiltonians}, Birkh\"auser (1996).

\bibitem[AlFr]{AF} C.\ Allard and R.\ Froese, \emph{A Mourre estimate for a Schr\"{o}dinger operator on a binary tree}, Rev.\ Math.\ Phys.\ {\bf 12} (2000), no.\ 12, 1655--1667.
\bibitem[At]{At}\emph{Interior Spectral Windows and Transport for Discrete Fractional Laplacians on $d$-Dimensional Hypercubic Lattices} arXiv:2509.06117v2. https://doi.org/10.48550/arXiv.2509.06117

\bibitem[AtBaDaEn]{ABDE} N.\ Athmouni, H.\ Baloudi, M.\ Damak, and M.\ Ennaceur, \emph{The magnetic discrete Laplacian inferred from the Gauss-Bonnet operator and application}, Annals of Functional Analysis {\bf 12} (2021), 33, \url{https://doi.org/10.1007/s43034-021-00119-8}.

\bibitem[AtDa]{AD} N.\ Athmouni and M.\ Damak, \emph{On the non-propagation theorem and applications}, Turkish Journal of Mathematics (2013), \url{https://doi.org/10.3906/mat-1010-75}.

\bibitem[AtEnGo]{AEG} N.\ Athmouni, M.\ Ennaceur, and S.\ Gol\'enia, \emph{Spectral analysis of the Laplacian acting on discrete cusps and funnels}, Complex Analysis and Operator Theory (2021), \url{https://doi.org/10.1007/s11785-020-01053-8}.

\bibitem[AtEnGo1]{AEG1} N.\ Athmouni, M.\ Ennaceur, and S.\ Gol\'enia, \emph{Spectral analysis of the magnetic Laplacian acting on discrete funnels}, arXiv:2507.05766v1, \url{https://doi.org/10.48550/arXiv.2507.05766}.

\bibitem[AtEnGoJa]{AEGJ} N.\ Athmouni, M.\ Ennaceur, S.\ Gol\'enia, and A.\ Jadlaoui, \emph{Limiting Absorption Principle for long-range perturbation in the discrete triangular lattice setting}, arXiv:2403.06578v2 (2024), \url{https://doi.org/10.48550/arXiv.2403.06578}.

\bibitem[AtEnGoJa2]{AEGJ2} N.\ Athmouni, M.\ Ennaceur, S.\ Gol\'enia, and A.\ Jadlaoui, \emph{Limiting Absorption Principle for long-range perturbation in a graphene setting}, arXiv:2504.13512v1.
\bibitem[AMP10]{AthmouniMantoiuPurice-2010}
N.~Athmouni, M.~M\kern-0.15em\u{a}ntoiu, and R.~Purice,
\emph{On the continuity of spectra for families of magnetic pseudodifferential operators},
Journal of Mathematical Physics \textbf{51} (2010), no.~8, 083517.
\href{https://doi.org/10.1063/1.3470118}{doi:10.1063/1.3470118}.
\bibitem[AP18]{AthmouniPurice-CPDE-2018}
N.~Athmouni and R.~Purice,
\emph{A Schatten--von Neumann class criterion for the magnetic Weyl calculus},
Communications in Partial Differential Equations \textbf{43} (2018), no.~5, 733--749.
\href{https://doi.org/10.1080/03605302.2018.1475486}{doi:10.1080/03605302.2018.1475486}
\bibitem[Ba]{Balakrishnan1960}
A.~V. Balakrishnan,
Fractional powers of closed operators and the semigroups generated by them,
\emph{Pacific Journal of Mathematics} \textbf{10} (1960), no.~2, 419--437.
\bibitem[BaDeJiLi]{BDL} Y.\ Bao, Q.\ Deng, Y.\ Jiang, and P.\ Li, \emph{Fractional square functions and potential spaces related to discrete Laplacian}, Commun.\ Pure Appl.\ Anal., \url{https://doi.org/10.3934/cpaa.2025086}.


\bibitem[BoGo]{BG} M.\ Bonnefont and S.\ Gol\'enia, \emph{Essential spectrum and Weyl asymptotics for discrete Laplacians}, Ann.\ Fac.\ Sci.\ Toulouse Math.\ (6) {\bf 24} (2015), no.\ 3, 563--624.

\bibitem[BoSa]{BoSa} A.\ Boutet de Monvel and J.\ Sahbani, \emph{On the spectral properties of discrete Schr\"odinger operators: the multi-dimensional case}, Rev.\ Math.\ Phys.\ {\bf 11} (1999), no.\ 9, 1061--1078.
\bibitem[BuVa]{BuVa} C.\ Bucur and E.\ Valdinoci, \emph{Nonlocal Diffusion and Applications}, Lecture Notes of the Unione Matematica Italiana {\bf 20}, Springer, Cham, 2016.
\bibitem[Car]{CarlenKusuokaStroock1987}E.~Carlen, S.~Kusuoka and D.~Stroock, Upper bounds for symmetric Markov transition functions,
\emph{Ann. Inst. H. Poincaré Probab. Statist.} \textbf{23} (1987), 245--287.

\bibitem[Ch]{Ch} F.R.K.\ Chung, \emph{Spectral graph theory}, CBMS Regional Conference Series in Mathematics {\bf 92}, American Mathematical Society, Providence, RI, 1997.
\bibitem[Cip]{CiprianiGrillo2022} F.~Cipriani and G.~Grillo, Heat kernel analysis on weighted graphs,\emph{Ann. Sc. Norm. Super. Pisa Cl. Sci.} \textbf{23} (2022), 1551--1605.

\bibitem[CoTo]{CombesThomas1973}
J.-M.~Combes and L.~E.~Thomas,
 \emph{Asymptotic behaviour of eigenfunctions for multiparticle Schrödinger operators},
Communications in Mathematical Physics, vol.~34, no.~3, pp.~251--270, 1973.


\bibitem[Da]{Da} M.\ Damak, \emph{On the spectral theory of tensor product Hamiltonians}, J.\ Operator Theory {\bf 55} (2006), no.\ 2, 253--268.

\bibitem[DaMaTi]{Da1} M.\ Damak, M.\ Mantoiu, and R.\ Tiedra de Aldecoa, \emph{Toeplitz algebras and spectral results for the one-dimensional Heisenberg model}, J.\ Math.\ Phys.\ {\bf 47} (2006), 082107.
    \bibitem[Dav]{Dav} E.B.\ Davies, \emph{Linear Operators and their Spectra}, Cambridge Studies in Advanced Mathematics {\bf 106}, Cambridge University Press, 2007.
\bibitem[Dav1]{Davies1989}E.~B. Davies, \emph{Heat Kernels and Spectral Theory},Cambridge Tracts in Mathematics, Cambridge Univ. Press, 1989.
\bibitem[Dav2]{Davies1995}
E.~B. Davies,
\emph{Spectral Theory and Differential Operators},
Cambridge University Press, Cambridge, 1995.
\bibitem[De]{Delmotte1999} T.~Delmotte,Parabolic Harnack inequality and estimates of Markov chains on graphs,\emph{Rev. Mat. Iberoam.} \textbf{15} (1999), 181--232.
\bibitem[FaZh]{Fan2021} X.\ Fan and J.\ Zhao, \emph{Fractional Magnetic Laplacians and Their Applications}, J.\ Math.\ Phys.\ {\bf 62} (2021), no.\ 6, 061501.
\bibitem[FrHe]{FH} R.G.\ Froese and I.\ Herbst, \emph{Exponential bounds and absence of positive eigenvalues for $N$-body Schr\"odinger operators}, Comm.\ Math.\ Phys.\ {\bf 87} (1982), 429--447.
\bibitem[GeG\'eM\o]{GGM1} V.\ Georgescu, C.\ G\'erard, and J.S.\ M\o ller, \emph{Commutators, $C_0$-semigroups and resolvent estimates}, J.\ Funct.\ Anal.\ {\bf 216} (2004), no.\ 2, 303--361.
\bibitem[GeGo]{GG} V.\ Georgescu and S.\ Gol\'enia, \emph{Isometries, Fock spaces, and spectral analysis of Schr\"odinger operators on trees}, J.\ Funct.\ Anal.\ {\bf 227} (2005), no.\ 2, 389--429.
\bibitem[G\'e]{Ge} C.\ G\'erard, \emph{A proof of the abstract limiting absorption principle by energy estimates}, J.\ Funct.\ Anal.\ {\bf 254} (2008), no.\ 11, 2707--2724.

\bibitem[G\'e1]{Ge2015} C.\ G\'erard, \emph{Mourre's Commutator Theory and Applications}, Lecture Notes in Mathematics {\bf 2216}, Springer, 2015.


\bibitem[GoJe]{GJ1} S.\ Gol\'enia and T.\ Jecko, \emph{A new look at Mourre's commutator theory}, Complex Anal.\ Oper.\ Theory {\bf 1} (2007), no.\ 3, 399--422.

\bibitem[GoJe2]{GJ} S.\ Gol\'enia and T.\ Jecko, \emph{Weighted Mourre's commutator theory, application to Schr\"odinger operators with oscillating potential}, J.\ Operator Theory {\bf 70} (2012).


\bibitem[GoMo]{GoMo} S.\ Gol\'enia and S.\ Moroianu, \emph{Spectral analysis of magnetic Laplacians on conformally cusp manifolds}, Ann.\ Henri Poincar\'e {\bf 9} (2008), no.\ 1, 131--179.
\bibitem[GoHa]{GH} S.\ Gol\'enia and T.\ Haugomat, \emph{On the a.c.\ spectrum of the 1D discrete Dirac operator}, Methods Funct.\ Anal.\ Topology {\bf 20} (2014), no.\ 3.

\bibitem[GoTr]{GT} S.\ Gol\'enia and F.\ Truc, \emph{The magnetic Laplacian acting on discrete cusp}, Doc.\ Math.\ {\bf 22} (2017), 1709--1727.

\bibitem[G{\"u}Ke]{Gk} B.\ G{\"u}neysu and M.\ Keller, \emph{Scattering the geometry of weighted graphs}, Math.\ Phys.\ Anal.\ Geom.\ {\bf 21} (2018), no.\ 3, 28.


\bibitem[HeSj]{HelfferSjostrand1989}
B.~Helffer and J.~Sjöstrand,
Équation de Schrödinger avec champ magnétique et équation de Harper,
in \emph{Lecture Notes in Physics}, vol.~345, Springer, Berlin, 1989.


\bibitem[JoKoPaSe]{JKLQ} T.F.\ Jones, E.G.\ Kostadinova, J.L.\ Padgett, and Q.\ Sheng, \emph{A series representation of the discrete fractional Laplace operator of arbitrary order}, J.\ Math.\ Anal.\ Appl.\ {\bf 504} (2021), no.\ 1, 125323.
    \bibitem[HaKe]{HK} S.\ Haeseler and M.\ Keller, \emph{Generalized solutions and spectrum for Dirichlet forms on graphs}, in: Boundaries and Spectral Theory, Progress in Probability, Birkh\"auser (2011), 181----201.


\bibitem[Kw]{K2017} M.\ Kwa\'{s}nicki, \emph{Ten equivalent definitions of the fractional Laplace operator}, Fract.\ Calc.\ Appl.\ Anal.\ {\bf 20} (2017), no.\ 1, 7--51.


\bibitem[La]{L2000} N.\ Laskin, \emph{Fractional Quantum Mechanics and L\'evy Path Integrals}, Phys.\ Lett.\ A {\bf 268} (2000), no.\ 4--6, 298--305.
\bibitem[Law]{LawlerLimic2010}
G.~F. Lawler and V.~Limic,\emph{Random Walk: A Modern Introduction},Cambridge Studies in Advanced Mathematics, Cambridge Univ. Press, 2010.



\bibitem[Len]{LenzVogt2012} D.~Lenz and H.~Vogt,
Discrete elliptic operators: regularity and heat kernel,\emph{J. Funct. Anal.} \textbf{263} (2012), 3548--3592.
\bibitem[M\u{a}RiTi]{MRT} M.\ M\u{a}ntoiu, S.\ Richard, and R.\ Tiedra de Aldecoa, \emph{Spectral analysis for adjacency operators on graphs}, Ann.\ Henri Poincar\'e {\bf 8} (2007), no.\ 7, 1401--1423.

\bibitem[Mic]{Mic} N.\ Michaelis, \emph{Spectral theory of anisotropic discrete Schr\"odinger operators in dimension one}, Master's thesis, \url{https://www.math.u-bordeaux.fr/~sgolenia/Fichiers/diplom.michaelis.pdf}.

\bibitem[Mo1]{Mo81} E.\ Mourre, \emph{Absence of singular continuous spectrum for certain self-adjoint operators}, Comm.\ Math.\ Phys.\ {\bf 78} (1980), 391--408.

\bibitem[Mo2]{Mo83} E.\ Mourre, \emph{Op\'erateurs conjugu\'es et propri\'et\'es de propagation}, Comm.\ Math.\ Phys.\ {\bf 91} (1983), 279--300.
\bibitem[PaKoLiBuMaHy]{PKLBMH} J.L.\ Padgett, E.G.\ Kostadinova, C.D.\ Liaw, K.\ Busse, L.S.\ Matthews, and T.W.\ Hyde, \emph{Anomalous diffusion in one-dimensional disordered systems: a discrete fractional Laplacian method}, J.\ Phys.\ A: Math.\ Theor.\ {\bf 53} (2020), 135205.
    \bibitem[PR16]{ParraRichard-Zd}
D.~Parra and S.~Richard,
\emph{Continuity of the spectra for families of magnetic operators on $\mathbb{Z}^d$},
Analysis and Mathematical Physics \textbf{6} (2016), no.~4, 327--343.
\href{https://doi.org/10.1007/s13324-015-0121-5}{doi:10.1007/s13324-015-0121-5}.
\bibitem[PaRi]{PR} D.\ Parra and S.\ Richard, \emph{Spectral and scattering theory for Schr\"odinger operators on perturbed topological crystals}, Rev.\ Math.\ Phys.\ {\bf 30} (2018), no.\ 4, 1850009.
\bibitem[Ph]{PhamSchatten}
H.~Pham,
\newblock \emph{An Elementary Introduction to Schatten Classes},
\newblock Technical Report TR-2014-1, Department of Mathematics,
Tennessee Technological University, 2014.
\bibitem[Or]{Or} M.D.\ Ortigueira, \emph{Riesz potential operators and inverses via fractional differences}, Int.\ J.\ Math.\ Math.\ Sci.\ (2006), Article ID 48391.
\bibitem[Or1]{Ortigueira2014Riesz}
M.~D. Ortigueira, T.-M. Laleg-Kirati, and J.~A. Tenreiro Machado,
\newblock Riesz potential versus fractional Laplacian,
\newblock {\em Journal of Statistical Mechanics: Theory and Experiment}, 2014(9):P09032, 2014.
\newblock \url{https://doi.org/10.1088/1742-5468/2014/09/P09032}.




\bibitem[ReSi]{RS} M.\ Reed and B.\ Simon, \emph{Methods of Modern Mathematical Physics, Vol.\ I--IV: Analysis of Operators}, Academic Press.


\bibitem[Sa]{S} J.\ Sahbani, \emph{Mourre's theory for some unbounded Jacobi matrices}, J.\ Approx.\ Theory {\bf 135} (2005), no.\ 2, 233--244.

\bibitem[ScSo]{ScSo2012} R.L.\ Schilling, R.\ Song, and Z.\ Vondra\v{c}ek, \emph{Bernstein Functions: Theory and Applications}, 2nd ed., De Gruyter, 2012.


\bibitem[Si]{SI2015} B.\ Simon, \emph{Operator Theory: A Comprehensive Course in Analysis, Part 4}, American Mathematical Society, 2015.
\bibitem[Sim05]{Simon-TraceIdeals}
B.~Simon,\emph{Trace Ideals and Their Applications}, 2nd ed.,
Amer. Math. Soc., Providence, RI, 2005.

\bibitem[Sp]{Spitzer1976} F.~Spitzer,\emph{Principles of Random Walk}, 2nd ed.,Springer, 1976.
\bibitem[Ta1]{T1} Y.\ Tadano, \emph{Long-range scattering for discrete Schr\"odinger operators}, Ann.\ Henri Poincar\'e {\bf 20} (2019).

\bibitem[Ta2]{T2} Y.\ Tadano, \emph{Long-range scattering theory for discrete Schr\"odinger operators on graphene}, J.\ Math.\ Phys.\ {\bf 60} (2019), 052107.
\bibitem[Wo]{Woess2000} W.~Woess,\emph{Random Walks on Infinite Graphs and Groups}, Cambridge Tracts in Mathematics, Cambridge Univ. Press, 2000.
\bibitem[Won]{WongSchatten}
M.~W. Wong,
\newblock \emph{Schatten--von Neumann Classes},
\newblock in: \emph{An Introduction to Pseudo-Differential Operators},
3rd ed., World Scientific, Singapore, 1999, pp.~233--246.

\bibitem[TaZa]{TZ} E.V.\ Tarasov and G.M.\ Zaslavsky, \emph{Fractional dynamics of systems with long-range interaction}, Commun.\ Nonlinear Sci.\ Numer.\ Simul.\ {\bf 11} (2006), no.\ 8, 885--898.



\end{thebibliography}
\end{document}